\documentclass[12pt,oneside]{amsart}
\title{A new definition of rough paths on manifolds}
\date{}
   \author{Youness Boutaib}
   \thanks{The first author carried most of this research while at the University of Oxford. The same author is thankful for the support of the DFG through the research unit FOR2402 in Berlin and Potsdam where important improvements on this work have been made.}
   \address{RWTH Aachen University. Chair for Mathematics of Information Processing. Pontdriesch 10, 52062 Aachen. Germany.}
   \email{boutaib.youness@gmail.com}
 	\author{Terry Lyons}
   \address{Mathematical Institute. University of Oxford. Andrew Wiles Building. Radcliffe Observatory Quarter, Woodstock Road. Oxford OX2 6GG. United Kingdom.}
   \email{Terry.Lyons@maths.ox.ac.uk}

\usepackage{amssymb} % necessaire pour les mathbb 
\usepackage{amsfonts}
\usepackage{amsmath} %pour underset
\usepackage{shuffle} %pour le shuffle product notation
\usepackage{amsthm} %pour des theoremes unnumbered
\usepackage[margin=1in]{geometry}  % set the margins to 1in on all sides
\usepackage[usenames,dvipsnames]{xcolor} % for colours

\newtheorem{theo}{Theorem}[section]
\newtheorem{Cor}[theo]{Corollary}
\newtheorem{Def}[theo]{Definition}
\newtheorem{example}[theo]{Example}
\newtheorem{examples}[theo]{Examples}
\newtheorem{lemma}[theo]{Lemma}
\newtheorem{notation}[theo]{Notation}
\newtheorem{Prop}[theo]{Proposition}
\newtheorem{Rem}[theo]{Remark}

\begin{document}
\maketitle
	\begin{abstract} Smooth manifolds are not the suitable context for trying to generalize the concept of rough paths on a manifold. Indeed, when one is working with smooth maps instead of Lipschitz maps and trying to solve a rough differential equation, one loses the quantitative estimates controlling the convergence of the Picard sequence. Moreover, even with a definition of rough paths in smooth manifolds, ordinary and rough differential equations can only be solved locally in such case. In this paper, we first recall the foundations of the Lipschitz geometry, introduced in \cite{CLLy}, along with the main findings that encompass the classical theory of rough paths in Banach spaces. Then we give what we believe to be a minimal framework for defining rough paths on a manifold that is both less rigid than the classical one and emphasized on the local behaviour of rough paths. We end by explaining how this same idea can be used to define any notion of coloured paths on a manifold.
	\end{abstract}
	%%%%%%%%%%%%%%%%%%%%%%%%%%%%%%%%%%%%%%%%%%%%%%%%%%%%%%%%
	%%%%%%%%%%%%%%%%%%%%%%%%%%%%%%%%%%%%%%%%%%%%%%%%%%%%%%%%
	%%
	% Introduction and motivation
	%%
	%%%%%%%%%%%%%%%%%%%%%%%%%%%%%%%%%%%%%%%%%%%%%%%%%%%%%%%%
	%%%%%%%%%%%%%%%%%%%%%%%%%%%%%%%%%%%%%%%%%%%%%%%%%%%%%%%%
	\section{Introduction} The theory of rough paths (Lyons, \cite{Lyons}) and its variations (e.g. Gubinelli, \cite{Gubinelli}) generalise Young's integration theory in a way that it separates the probabilistic and deterministic parts of the strongly probabilistic It\^o calculus, using only the variation of paths as a way of measuring their smoothness. A crowning achievement of the theory is understanding that, as far as ordinary differential equations are concerned, a path should not be defined by its graph but rather be identified as a choice of its signature (that is, the sequence of its iterated integrals). For example, the signature of a Brownian motion could be calculated using either It\^o's or Stratonovich's calculus and it is this choice that leads one to solve a stochastic differential equation driven by a Brownian motion either in the sense of It\^o's calculus or Stratonovich's. More generally, the theory of rough paths provides us with a deterministic calculus constructed in a way that it does not depend intrinsically on how a signature is defined but rather on common algebraic (that can be summerized by a Hopf algebra structure) and analytical properties that all signatures are expected to satisfy. These works have, in particular, enriched the toolbox of stochastic analysis with deterministic -path by path- results and widened its scope to rougher paths than the Brownian motion. The underlying philosophy of the theory also opened the door for solving a certain class of Stochastic Partial Differential Equations that require making sense of classically ill-defined products of distributions. This was carried incrementally through the development of alternative rough path theories: branched rough paths \cite{Gubinelli2}, para-controlled calculus \cite{GIP} and regularity structures \cite{Hairer} (see \cite{FH} for a succinct exposition that goes from the theory of rough paths to that of regularity structures.) In addition to the natural applications that come with stochastic analyis, the theory of rough paths highlighted the role of signatures as highly efficient transforms of paths \cite{BGLY,HL}, which led to their exploitation in recent works in machine learning (the literature on the subject being abundant, we cite \cite{AGGLK,CNO,Graham} as varied use-case examples).
	
	While the need for such calculus on manifolds arises naturally from both the points of view of pure and applied mathematics (see for example the introductions to \cite{CLLy, Lee} for the motivation), the literature on rough paths on manifolds is still very limited compared to its counterpart in the Euclidean setting or even to that of stochastic analysis on manifolds. The main attempts to generalise these notions to manifolds are due to Cass, Litterer and Lyons in \cite{CLLy} in the framework of what is called a Lipschitz manifold and Driver and Semko in \cite{DS} (for paths controlled by rough paths on Riemannian manifolds). We also refer to Cass, Driver and Litterer in \cite{CDL} (for weakly geometric rough paths on submanifolds embedded in the Euclidean space) and Bailleul \cite{Bailleul}. As in \cite{CLLy}, we aim to avoid to put too much structure on the manifold we work on or to exactly mimic the construction of rough paths on the Euclidean space in the non-linear framework of a manifold. In order to show how one can simply translate the notion of rough paths (and in general that of \emph{coloured paths}) to manifolds, we use below the language of category theory to introduce local Lipschitz manifolds and explain how this can be done without any further particular considerations of the class of manifolds one is working on (for example, one does not need the manifold to be Riemannian to be able to measure the smoothness of paths).		
	%%%%%%%%%%%%%%%%%%%%%%%%%%%%%%%%%%%%%%%%%%%%%%%%%%%%%%%%
	%%%%%%%%%%%%%%%%%%%%%%%%%%%%%%%%%%%%%%%%%%%%%%%%%%%%%%%%
	%%
	% Review of key elements in the theory of the rough paths
	%%
	%%%%%%%%%%%%%%%%%%%%%%%%%%%%%%%%%%%%%%%%%%%%%%%%%%%%%%%%
	%%%%%%%%%%%%%%%%%%%%%%%%%%%%%%%%%%%%%%%%%%%%%%%%%%%%%%%%
	\section{Review of key elements in the theory of the rough paths}
	We start by defining rough paths and giving a reminder of all the notions that will be necessary to us in the rest of this work,  one of which will be the extension of the notion of Lipschitz (H\"older) maps. Most of the results recalled in this section can be easily found, for example, in \cite{CLL}, \cite{Lyons} and \cite{LyQ}. We mostly give some examples and prove some of the statements to familiarize the reader with any possible new notions. In particular, we will attach a certain importance to associating a starting point or a trace to geometric rough paths.
	%%%%%%%%%%%
	%%
	% Signatures of paths
	%%
	%%%%%%%%%%%
	\subsection{Signatures of paths}
	%%
	% The concept of the $p$-variation
	%%
	\subsubsection{The concept of the $p$-variation}\label{sec:PVarConcept}
	%%%%	
	%%%% DEF: p-variation
	%%%%
	We introduce the notion of $p$-variation, crucial in both Stieltjes's and Young's integration theories and the rough path theory. It is a way of measuring the amplitude of the oscillations of a path, independently of when said oscillations occur.
    \begin{Def}
    Let $p\geq1$, $(E,\|.\|)$ be a normed vector space and $T \geq 0$. Let $x: [0,T]\to E$ be a continuous path. For a finite subdivision $D=(t_i)_{0\leq i\leq n}$ of $[0,T]$, we denote by $\|x\|_{p,D}$ the quantity:
    \[ \|x\|_{p,D}:=\left(\sum_{i=0}^{n-1}\|x_{t_{i+1}}-x_{t_i}\|^p\right)^{\frac{1}{p}}=\left(\sum_{D}\|x_{t_{i+1}}-x_{t_i}\|^p\right)^{\frac{1}{p}}\]
    $x$ is said to have a finite $p$-variation over $[0,T]$ if $\{\|x\|_{p,D}|D\in\mathcal{D}_{[0,T]}\}$ has a finite supremum. In this case, this supremum is called the $p$-variation of $x$ over $[0,T]$ and is denoted by $\|x\|_{p,[0,T]}$. When $p=1$, we say that the path $x$ has bounded variation.
    \end{Def}  
    \begin{Rem}
    The previous definition can easily be adapted for a metric space $(E,d)$ but we will not use this generalisation in the rest of these notes.
    \end{Rem}
   	%%%%	
	%%%% EX: paths with finite p-variation
	%%%%
    \begin{examples} Let $T\geq 0$: \begin{enumerate}
	\item Let $E$ be a normed vector space. A continuously differentiable $E$-valued path $x$ over $[0,T]$ is of bounded variation over $[0,T]$ and $\|x\|_{1,[0,T]}\leq T\|\dot{x}\|_{\infty,[0,T]}$.
	\item Let $E$ be a normed vector space and $0<\alpha\leq1$. An $\alpha$-H\"older $E$-valued path over $[0,T]$ has finite $1/\alpha$-variation over $[0,T]$.
	\item Let $(\Omega, \mathcal{A}, \mathbb{P})$ be a probability space and $B$ be a Brownian motion over $[0,T]$: \begin{itemize}
	\item $B$ has finite $p$-variation over $[0,T]$ almost surely, for any $p>2$.
	\item Seen as an $L^2(\Omega)$-valued path, $B$ has finite $2$-variation.
	\end{itemize}
	\end{enumerate}    \end{examples}
	%%%%	
	%%%% PROP: Spaces of p-variation
	%%%%
	 \begin{Prop} \cite{CLL,CG}
    Let $p\geq1$ and $(E,\|.\|)$ be a Banach space.
    \begin{itemize}
      \item The set $\mathcal{V}^p([0,T],E)$ of all continuous paths from $[0,T]$ to $E$ that have a finite $p$-variation over $[0,T]$ is a vector space when endowed with the natural operations of addition and multiplication by a scalar.
      \item The map:
        \begin{displaymath}
        \begin{array}{crcl}
          \|.\|_{\mathcal{V}^p([0,T],E)}:&\mathcal{V}^p([0,T],E)&\to&\mathbb{R}^+\\
           &x&\mapsto&\|x\|_{p,[0,T]}+\sup\limits_{t\in [0,T]}\|x_t\|\\
         \end{array}
        \end{displaymath}
      defines a norm on $\mathcal{V}^p([0,T],E)$ called the $p$-variation norm.
      \item $(\mathcal{V}^p([0,T],E),\|.\|_{\mathcal{V}^p([0,T],E)})$ is a Banach space.
      \item $\forall q\geq p\geq 1: \mathcal{V}^p([0,T],E) \subseteq \mathcal{V}^q([0,T],E)$.
    \end{itemize}
    \end{Prop}
	The manipulation of $p$-variations is often made easier by the introduction of controls. For a compact interval $J$, we will denote by $\Delta_{J}$ the simplex of all pairs $(s,t)\in J^2$ such that $s\leq t$.
	%%%%	
	%%%% DEF: Controls
	%%%%
    \begin{Def}
    A function $\omega:\Delta_{[0,T]} \to \mathbb{R}_+$ is said to be a control if it has the following properties:
    \begin{itemize}
      \item $\omega$ is continuous.
      \item $\omega$ is super-additive i.e. $\omega(s,u)+ \omega(u,t) \leq \omega(s,t) \quad ,\forall \, 0\leq s\leq u \leq t\leq T$.
    \end{itemize}
    \end{Def}
    Lemma \ref{NatControl} shows that to every path of finite $p$-variation, we can associate a ``natural" control.
	%%%%	
	%%%% Natural Controls of p-paths
	%%%%
    \begin{lemma}\label{NatControl} \cite{CLL,CG,LyQ}
    Let $p\geq1$, $E$ be a normed vector space and $T\geq 0$. Let $x: [0,T]\to E$ be a continuous path of finite $p$-variation over $[0,T]$. Then the function $\omega$ defined on $\Delta_{[0,T]}$ by $\omega (s,t)=\|x\|_{p,[s,t]}^p$ is a control.
    \end{lemma}  
	Conversely, if we can find a control that bounds from above the $p$\textsuperscript{th} power of the increments of a continuous path then this path is necessarily of finite $p$-variation. This will be the content of theorem \ref{Controlled}. Consequently, it also gives a criterion for the finiteness of the $p$-variation of a path without having to go back to the definition.
	%%%%	
	%%%% p-paths are controlled
	%%%%
    \begin{theo}\label{Controlled} \cite{CLL,CG,LyQ}
    Let $p\geq1$, $E$ be a normed vector space and $T\geq 0$. Let $x: [0,T]\to E$ be a continuous path. There exists a control $\omega$ defined over $\Delta_{[0,T]}$ such that for every $(s,t) \in \Delta_{[0,T]}$ we have: $\|x_t-x_s\|^p\leq \omega (s,t)$ if and only if $x$ has a finite $p$-variation. In this case, and for such a control $\omega$, we have: 
	\[\forall (s,t) \in \Delta_{[0,T]}: \, \|x\|_{p,[s,t]}^p\leq \omega (s,t)\]
	We say in this case that the $p$-variation of $x$ is controlled by $\omega$.
    \end{theo}
	%%%%%%%%%%%
	%%
	% The signature of a path
	%%
	%%%%%%%%%%%
	\subsubsection{The signature of a path}\label{sec:signatures}
	We start by recalling the definition of the tensor algebra of a vector space (see \cite{Reutenauer} for an exhaustive exposition).
	\begin{Def}
    Let $E$ be a vector space. For every $n\in \mathbb{N}^*$, let $E^{\otimes n}$ be the space of homogeneous tensors of $E$ of degree $n$. We use the convention: $E^{\otimes 0}=\mathbb{R}$. The set of formal series of tensors of $E$, denoted by $T((E))$ is defined by the following:
    \[ T((E)):=\{\mathbf{a}=(a_n)_{n\in \mathbb{N}}|\forall n\in \mathbb{N}: \quad a_n \in E^{\otimes n}\}\]
    $T((E))$ has an algebra structure when endowed with the operations defined by the following: for $\mathbf{a}=(a_n)_{n\in \mathbb{N}}$ and $\mathbf{b}=(b_n)_{n\in \mathbb{N}}$ in $T((E))$ and $\lambda \in \mathbb{R}$:
    \begin{description}
      \item [(Addition)] $\mathbf{a}+\mathbf{b}=(a_n+b_n)_{n \in \mathbb{N}}$,
      \item [(Multiplication)] $\mathbf{a} \otimes \mathbf{b}=\left(\sum\limits_{k=0}^{n} a_k \otimes b_{n-k}\right)_{n \in \mathbb{N}}$,
      \item [(Multiplication by a scalar)] $\lambda . \mathbf{a}=(\lambda a_n)_{n \in \mathbb{N}}$.
    \end{description}
    \end{Def}
    The invertible elements of the tensor algebra can be easily identified:
	%%%%	
	%%%% Invertible tensors
	%%%%
    \begin{Prop}\label{InvertibleTensors}\cite{CLL}
    Let $E$ be a real vector space. $(T((E)),+,.,\otimes)$ is a non-commutative (assuming dim($E$)$ \geq 2$) unital algebra with unit $\mathbf{1}=(1,0,0,\ldots)$. An element $\mathbf{a}=(a_n)_{n\in \mathbb{N}}$ in $T((E))$ is invertible if and only if $a_0 \neq 0$. In this case, its inverse is given by the well-defined series:
    \[\mathbf{a}^{-1}=\frac{1}{a_0}\sum_{n \geq 0}\left(\mathbf{1}-\frac{\mathbf{a}}{a_0}\right)^n\]
    \end{Prop}
    \begin{notation}
    \[\widetilde{T}((E))=\{\mathbf{a}=(a_n)_{n\in \mathbb{N}} \in T((E))| \quad a_0=1\}\]
    \end{notation}
	We now proceed to the definition of the signature of a path. This is given by the sequence of its iterated integrals and is a well studied subject since \cite{Chen}:
	%%%%	
	%%%% DEF: Signature
	%%%%
	\begin{Def}
	Let $E$ be a Banach space and $T \geq 0$. Let $x:[0,T]\to E$ be a path of bounded variation. For $(s,t) \in \Delta _{[0,T]}$, we define the following sequence by induction (well-defined by Stieltjes' integration theory):
    \[ \left \{ \begin{array}{ll}
      S^0(x)_{(s,t)}=1& \\
      S^n(x)_{(s,t)}=\int_{[s,t]}S^{n-1}(x)_{(s,u)} \otimes \mathrm{d}x_u & ,\forall n \in \mathbb{N}^* \\
      \end{array} \right.
    \]
	For every pair $(s,t) \in \Delta _{[0,T]}$, the sequence $(S^n(x)_{(s,t)})_{n\in \mathbb{N}}$, simply denoted $S(x)_{(s,t)}$, is called the signature of $x$ over $[s,t]$. For $N\in \mathbb{N}^*$, $(S^n(x)_{(s,t)})_{n\leq N}$, simply denoted $S_{N}(x)_{(s,t)}$, is called the truncated signature of $x$ over $[s,t]$ of degree $N$.
    \end{Def}
	The label ``signature'', implicitely implying the full characterization of a path, can be justified at many levels:
	\begin{itemize}
	\item Lyons and Hambly in \cite{HL} show that the signature of a path of bounded variation fully and uniquely characterizes the path in question up to a tree-like equivalence. This result has been extended later in more general contexts. Let us cite in particular the work of Le Jan and Qian \cite{LQ} for sample paths of the Brownian motion and that of Boedihardjo, Geng, Lyons and Yang \cite{BGLY} for (weakly) geometric rough paths (that we will introduce later).
	\item In the context of differential equations, the signature of the control signal is the only needed information to get the solution. This can be shown for example easily and explicitely in the case where the vector fields in the differential equation are linear and continuous (details in \cite{CLL} for example).
	\end{itemize} 
	\begin{lemma}\cite{CLL}
	Let $E$ be a vector space. Let $m \in \mathbb{N}$ be an integer and define \[B_m=\{\mathbf{a}=(a_n)_{n\in \mathbb{N}}|\forall i \in \{0,\ldots,m\} \quad a_i=0\}\] Then is $B_m$ an ideal of $T((E))$.
	\end{lemma}
	%%%%	
	%%%% DEF: Truncated tensors
	%%%%
	\begin{Def}\cite{CLL}
	Let $E$ be a vector space. Let $m \geq 0$ be an integer. The truncated tensor algebra of order $m$ of $E$, denoted by $T^{(m)}(E)$, is the quotient algebra $T((E))/B_m$. We will denote the canonical homomorphism $T((E))\to T((E))/B_m$ by $\pi_m$. There is a natural identification between $T^{(m)}(E)$ and $\bigoplus\limits_{0\leq i\leq m} E^{\otimes i}$.
	\end{Def}
	%%%%	
	%%%% DEF: Permutation of tensors
	%%%%
	\begin{Def}[Action of the Symmetric Group on Tensors]\cite{Reutenauer}
	Let $n \in \mathbb{N}^*$, $\sigma \in \mathcal{S}_n$ and $E$ be a vector space. We define the action of $\sigma$ on the homogenous tensors of $E$ of order $n$ as a linear map by the following: 
	\[\forall x_1,x_2,\ldots, x_n \in E \quad \sigma(x_1\otimes x_2\otimes \cdots \otimes x_n)=x_{\sigma(1)}\otimes x_{\sigma(2)}\otimes \cdots \otimes x_{\sigma(n)}\]
	\end{Def}
    \begin{Def}
    Let $E$ be a vector space.
    A norm $\|.\|$ on $T((E))$ is said to be admissible if the two following properties hold:
    \begin{enumerate}
      \item $\forall n \in \mathbb{N}^*,\,\forall \sigma \in \mathcal{S}_n,\, \forall x \in E^{\otimes n} \quad \| \sigma x\|=\|x\|$.
      \item $\forall n,m \in \mathbb{N}^*,\, \forall x \in E^{\otimes n},\, \forall y \in E^{\otimes m} \quad \| x \otimes y\|\leq\|x\| \|y\|$.
    \end{enumerate}
    \end{Def}
	 We will assume in the rest of this paper that $(E^{\otimes n})_{n\geq 1}$ are endowed with admissible norms. A discussion on certain basic properties of norms on tensor product spaces can be found for example in \cite{Boutaib}.
	%%%%%%%%%%%
	%%
	% Algebraic and analytic properties of the signature
	%%
	%%%%%%%%%%%
	\subsubsection{Some algebraic and analytic properties of signatures}
	We resume the notations of the previous subsection. Let $S$ (respectively $S_n$, for $n\in\mathbb{N}^*$) denote the map giving the signature (resp. the truncated signature of degree $n$) over $[0,T]$ of $E$-valued paths defined on $[0,T]$ that have bounded variation. We give below three basic properties of signatures that will be key in defining rough paths in the next subsection. For a vector space $E$, we identify linear forms on $T((E))$ with the elements of the vector space $T(E^*)$ (the tensor polynomials of $E^*$, $E^*$ being the space of linear forms on $E$). We start by what we will be calling the shuffle product property:
	%%%%	
	%%%% Shuffle product property
	%%%%
	\begin{lemma}\label{Shuffle}\cite{CLL}
	Let $E$ be a Banach space. For every two linear forms $e$ and $f$ defined on $T((E))$, there exists a linear form denoted by $e\shuffle f$ such that:
	\[\forall x \in \mathcal{V}^1([0,T],E) \quad e(S(x)).f(S(x))=(e\shuffle f)(S(x))\]
	$e\shuffle f$ is called the shuffle product of $e$ and $f$.
	\end{lemma}
	%%%%	
	%%%% DEF: Concatenation of paths
	%%%%
   \begin{Def}
    Let  $t \in [0,T]$ and $x: [0,t] \to E$ and $y: [t,T]\to E$ be any two paths. We call the concatenation of $x$ and $y$ the path denoted $x*y$ defined over $[0,T]$ by:
    \[\left \{ \begin{array}{ll}
    (x*y)_u=x_u&, \textrm{if} \quad u \in [0,t]\\
    (x*y)_u=y_u-y_t+x_t&, \textrm{if} \quad u \in [t,T]\\
    \end{array} \right.
    \]
    We define in a similar way the concatenation of any two paths defined on two adjacent segments.
   \end{Def}
	\begin{Rem}
	The concatenation operation is associative.
	\end{Rem}
	The second property that is of interest to us is what we call the multiplicativity property. The following theorem is due to Chen \cite{Chen2}. For a proof, see also \cite{CLL}:
	%%%%	
	%%%% THM: Signature of concatenations
	%%%%
   \begin{theo}\label{Chen1}
   Let  $t \in [0,T]$. Consider $x \in \mathcal{V}^1([0,t],E)$ and $y \in \mathcal{V}^1([t,T],E)$. Then $x*y\in \mathcal{V}^1([0,T],E)$ and:
	\[S(x*y)_{(0,T)}=S(x)_{(0,t)}\otimes S(y)_{(t,T)}\]
   \end{theo}
  \begin{Rem}\label{NoConcat}
	Theorem \ref{Chen1} is usually used in the following way (which is equivalent to the previous formulation): for every path $x \in \mathcal{V}^1([0,T],E)$ and  $s,u,t \in [0,T]$ such that $s\leq u \leq t$, we have:
	\[S(x)_{(s,t)}=S(x)_{(s,u)}\otimes S(x)_{(u,t)}\]
   \end{Rem}
	\begin{Rem}\label{SimpleSig}
	For $x \in \mathcal{V}^1([0,T],E)$ and for $(s,t) \in \Delta_{[0,T]}$, $S(x)_{(s,t)} \in \widetilde{T}((E))$ and is therefore invertible (proposition \ref{InvertibleTensors}). Consequently, using theorem \ref{Chen1}: 
\[S(x)_{(s,t)}=(S(x)_{(0,s)})^{-1}\otimes S(x)_{(0,t)}.\]
It is then equivalent to either study the path $u \mapsto S(x)_{(0,u)}$ on the interval $[0,T]$ or $(s,t)\mapsto S(x)_{(s,t)}$ on the simplex $\Delta_{[0,T]}$.
	\end{Rem}
	Finally, the norms of the successive terms in a signature of a path of bounded variation decay in a factorial way with repect to the $1$-variation:
	% Simple factorial decay
    \begin{theo}\label{DF0}\cite{Lyons2}
      Let $x$ be in $\mathcal{V}^1([0,T],E)$, then:\[
      \forall n \in \mathbb{N}^* \quad \forall (s,t) \in \Delta_{[0,T]} \quad \|S^n(x)_{(s,t)}\|\leq \frac{\|x\|^n_{1,[s,t]}}{n!}. \]
    \end{theo}
	%%%%%%%%%%%
	%%
	% Rough Paths
	%%
	%%%%%%%%%%%
	\subsection{Rough Paths}
	The theory of rough paths generalizes the concept of signatures to more irregular paths and provides the tools to solving differential equations driven by these without having to build a whole new theory of integration for each one of them (as in It\^o's calculus). The concept of rough paths finds its source in the signature and the main analytic and algebraic properties that it satisfies. The space of geometric rough paths is indeed simply defined as the completion of that of signatures of paths with bounded variation under a suitably chosen metric similar to the $p$-variation metric for paths introduced in section \ref{sec:PVarConcept}. We introduce here the basic definitions and constructions.
	%%
	% Multiplicative functionals
	%%
	\subsubsection{Multiplicative functionals}
	The appropriate higher order generalisation of the linearity of the integral for iterated integrals and of Chen's multiplicativity property for signatures is expressed in terms of multiplicative functionals:
	%%%%	
	%%%% DEF: Mult. functionals
	%%%%
	  	\begin{Def}\cite{CLL, Lyons}
	  Let $E$ be a normed vector space and $T\geq 0$. Let $X$ be a map on $\Delta_{[0,T]}$ with values in $T((E))$ (respectively in $T^{(n)}(E)$, with $n \in \mathbb{N}^*$). $X$ is said to be a multiplicative functional (resp. a multiplicative functional of degree $n$) if the following holds:
		\begin{enumerate}
  		\item $X$ is continuous.
  		\item $\forall t \in [0,T] \quad X_{(t,t)}=\mathbf{1}$.
  		\item $X$ is multiplicative: $\forall \,0\leq s \leq u \leq t \leq T \quad X_{(s,t)}=X_{(s,u)} \otimes X_{(u,t)}$.
  		\end{enumerate}
	\end{Def}
	\begin{Rem}
	When no confusion is possible, we may use the term multiplicative functional with no reference to its degree being finite or not.
	\end{Rem}
	\begin{Rem}
	We will use the notation $X^i$ for the component of $X$ of degree $i$.
	\end{Rem}
	It is clear, by Chen's theorem \ref{Chen1}, that, for every path $x \in \mathcal{V}^1([0,T],E)$, $S(x)$ is a multiplicative functional and that for every $n \in \mathbb{N}^*$, $S_n(X)$ is a multiplicative functional of degree $n$.
	%%
	% $p$-variation metric and and controls
	%%
	\subsubsection{$p$-variation metric and rough paths}
	We generalise now the notion of $p$-variation:
	%%%%	
	%%%% DEF: Mult. functionals
	%%%%
		\begin{Def}\label{VarMF}\cite{CLL, Lyons}
	Let $E$ be a normed vector space and $T\geq 0$. Let $p \geq 1$. Let $X$ be a map on $\Delta_{[0,T]}$ with values in $T((E))$ (respectively in $T^{(n)}(E)$, with $n \in \mathbb{N}^*$) and $\omega$ be a control over $[0,T]$. $X$ is said to have a finite $p$-variation over $[0,T]$ controlled by $\omega$ if:
  	\[\forall i\in \mathbb{N}^* \mathrm{(resp. }\forall i\in [\![1,n]\!]\mathrm{)}, \forall \,0 \leq s \leq t \leq T: \quad \|X^i_{(s,t)}\|\leq \frac{\omega (s,t)^{\frac{i}{p}}}{\beta_p \left( \frac{i}{p} \right)!} \]
  	where we write $x!$ for $\Gamma (x+1)$, with $\Gamma$ being the usual extension of the factorial (the Gamma function) and: 
 	\[\beta_p=p\left(1+\sum_{k=1}^{\infty}\left(\frac{2}{k}\right)^{\frac{[p]+1}{p}}\right)\]
  	If there exists a control such that the previous properties holds, we may say that $X$ has a finite $p$-variation over $[0,T]$ without mentioning the control. We denote by $\mathcal{C}_{0,p}(\Delta_{[0,T]},T^{[p]}(E))$ the set of continuous paths defined over $\Delta_{[0,T]}$ with values in $T^{[p]}(E)$ and that have finite $p$-variation.
	\end{Def}
	We see, that for $n=1$, the concept of $p$-variation in definition \ref{VarMF} is the same as the one introduced in subsection \ref{sec:PVarConcept}. By theorem \ref{DF0}, signatures have finite $1$-variation. The $p$-variation control substitutes then the factorial decay property of signatures.
	\begin{Rem}
	One can also easily note that for $1\leq q\leq p$, a multiplicative functional of finite $q$-variation is of finite $p$-variation.
	\end{Rem}
	%The next result (the neo-classical inequality) will be of a technical use to us in a subsequent section and is crucial for the extension theorem  for rough paths (see \cite{hara} for a complete proof or \cite{Lyons, LyQ} for a loose version of the inequality).
  	%Neoclassical inequality
	%\begin{lemma}[Neo-classical inequality]\label{NeoClassical}
  	%\[
  	%\forall p \geq 1 ,\, \forall n \in \mathbb{N} ,\, \forall a,b \in \mathbb{R}^+ \quad \frac{1}{p} \sum_{k=0}^{n}\frac{a^{\frac{k}{p}}b^{\frac{n-k}{p}}}{(\frac{k}{p})!(\frac{n-k}{p})!}\leq \frac{(a+b)^{\frac{n}{p}}}{(\frac{n}{p})!}
  	%\]
  	%\end{lemma}
	Lyons' extension theorem states that a multiplicative functional of finite $p$-variation is uniquely determined by its terms of degree less than or equal to $[p]$:
	%Uniqueness of rough paths
	%\begin{lemma}\label{Petit32}\cite{CLL,Lyons}
	%Let $p \geq1$ and $n \in \mathbb{N}^*$ such that $n\geq [p]$ and $X$ and $ Y$ be two multiplicative functionals (resp. multiplicative functionals of degree $n$) that have a finite $p$-variation over $[0,T]$. If $\pi_{[p]}(X)=\pi_{[p]}( Y)$, then $X= Y$.
	%\end{lemma}
	%%%%	
	%%%% THM: Extension theorem
	%%%%
	\begin{theo}[Extension theorem]\label{ExtThRP}\cite{CLL,Lyons} Let $p\geq1$ and $n \in \mathbb{N}^*\cup\{{\infty}\}$ such that $n\geq [p]$. Let $E$ be a Banach space. Let $X$ be a multiplicative functional of degree $[p]$ in $E$ that has a finite $p$-variation over $[0,T]$ controlled by a control function $\omega$. There exists a unique multiplicative functional $\tilde{X}$ of degree $n$ that has a finite $p$-variation over $[0,T]$ and such that $\pi_{[p]}(X)=\pi_{[p]}(\tilde{X})$. Furthermore, the $p$-variation of $\tilde{X}$ is also controlled by $\omega$.
	\end{theo}
	\begin{Rem} Theorem \ref{ExtThRP} partially encompasses Young's integration theory \cite{Young} in the sense that it allows the construction of the signature of a path of finite $p$-variation, when $p<2$, using only the increments of the path.
	\end{Rem}
	The signature of a path of bounded variation is, therefore, by theorem \ref{ExtThRP}, the only multiplicative functional with finite $1$-variation which component of degree $1$ corresponds to the increments of said path.
	
	We are now ready to give a definition for rough paths:
	%%%%	
	%%%% DEF: p-RPath
	%%%%
  \begin{Def}[Rough Paths] \cite{CLL,Lyons}
  Let $p\geq1$ and $E$ be a Banach space. A $p$-rough path is a multiplicative functional of degree $[p]$ that has finite $p$-variation. The set of all $p$-rough paths over $[0,T]$ with values in $E$ will be denoted by $\Omega_p([0,T];E)$, while that of $p$-rough paths over sub-segments of $[0,T]$ will be denoted by $\Omega_p^{[0,T]}(E)$.
  \end{Def}
	%%%%	
	%%%% DEF: p-var for p-RPath
	%%%%
  \begin{Def}
  We define on $\mathcal{C}_{0,p}(\Delta_{[0,T]},T^{[p]}(E))$ the $p$-variation metric, denoted by $\tilde{d}_p$, by the following:
 \[ \tilde{d}_p(X, Y)=\max_{0\leq i \leq [p]} \sup_{D \in \mathcal{D}_{[0,T]}} \left( \sum_D \|X_{(t_j,t_{j+1})}^i- Y_{(t_j,t_{j+1})}^i \|^{\frac{p}{i}} \right)^{\frac{1}{p}} \]
	for all $X,  Y \in \mathcal{C}_{0,p}(\Delta_{[0,T]},T^{[p]}(E))$.\\
	(For a subdivision $D=(t_j)_{0\leq j \leq n}$, we write $\sum\limits_D a_{(t_j,t_{j+1})}:=\sum\limits_{j=0}^{n-1} a_{(t_j,t_{j+1})}$.)
  \end{Def}
	%\begin{Def} Let $p\geq1$ and $E$ be a Banach space. Let $X\in\mathcal{C}_{0,p}(\Delta_{[0,T]},T^{[p]}(E))$ and $(X(n))_{n\in\mathbb{N}}$ be a sequence of elements of $\mathcal{C}_{0,p}(\Delta_{[0,T]},T^{[p]}(E))$. We say that $(X(n))_{n\in\mathbb{N}}$ converges to $X$ in the $p$-variation topology, if there exists a control function $\omega$ and a real vanishing sequence $(a(n))_{n\in\mathbb{N}}$ such that:
	%\begin{enumerate}
	%\item $\omega$ controls the $p$-variation of $X$ and of $X(n)$ for each $n\in\mathbb{N}$.
	%\item $\forall n\in\mathbb{N}, \forall (s,t)\in\Delta_{[0,T]}, \forall i\in [\![1,[p]]\!]: \|X^{i}_{s,t}-X^{i}_{s,t}(n)\|\leq a(n)\omega(s,t)^{i/p}$.
	%\end{enumerate}
	%\end{Def}
	%Convergence in the $p$-variation topology and the $p$-variation metric are ``almost" equivalent in the following way:
	%%%%	
	%%%% Prop: Convergence in pVar topology 
	%%%%
	%\begin{Prop}\label{PVARTOPMET}\cite{CLL,LyQ} Let $p\geq1$ and $E$ be a Banach space. Let $X\in\mathcal{C}_{0,p}(\Delta_{[0,T]},T^{[p]}(E))$ and $(X(n))_{n\in\mathbb{N}}$ be a sequence of elements of $\mathcal{C}_{0,p}(\Delta_{[0,T]},T^{[p]}(E))$. 
	%\begin{itemize} \item if $(X(n))_{n\in\mathbb{N}}$ converges to $X$ in the $p$-variation metric, then the convergence also occurs in the $p$-variation topology.
	 %\item if $(X(n))_{n\in\mathbb{N}}$ converges to $X$ in the $p$-variation topology, then there exists a subsequence of $(X(n))_{n\in\mathbb{N}}$ converging to $X$ in the $p$-variation metric.
	%\end{itemize}
	%\end{Prop}
	Let us note the following important property:
	%%%%	
	%%%% THM: omega is complete
	%%%%
  \begin{theo}\label{OmegaComplete}\cite{LyQ}
  $(\Omega_p([0,T];E),\tilde{d}_p)$ is a complete metric space.
  \end{theo} 
	%%
	% geometric $p$-rough paths
	%%
	\subsubsection{Geometric $p$-rough paths} We introduce now geometric rough paths, a central notion in our paper. For more details, see \cite{CLL} or \cite{Lyons} (or \cite{FV} for an extensive study of the subject in the finite-dimensional case motivated with examples from stochastic analysis).
	%%%%	
	%%%% DEF: Geometric rough paths 
	%%%%
	\begin{Def}\label{DefGeoRP} Let $T\geq 0$ and $E$ be a real Banach space. We define the set of $E$-valued geometric $p$-rough paths over $[0,T]$ to be the closure of the set $\{S_{[p]}(x)|x\in \mathcal{V}^1([0,T],E)\}$ in the $p$-variation metric. It is denoted by $G\Omega_p([0,T];E)$. The set of $E$-valued geometric $p$-rough paths over  sub-segments of $[0,T]$ will be denoted $G\Omega_p^{[0,T]}(E)$.
  \end{Def}
	Rough paths do not take into account the shuffle product property satisfied by signatures. Geometric rough paths do not have this ``setback" by continuity of linear forms:
	%Conservation of the algebraic and analytic properties
  \begin{Prop}\label{Cons}
  Let $E$ be a Banach space and $T\geq 0$. Let $X$ be a geometric $p$-rough path over $[0,T]$ in $E$.
  Then $X$ is an element of $\Omega_p^{[0,T]}(E)$ (in particular, $X$ is multiplicative and has finite $p$-variation) and for every $(s,t)\in \Delta_{[0,T]}$, $X_{(s,t)}$ satisfies the shuffle product property.
  \end{Prop}
	\begin{Rem}
	Let $(G^{[p]},\otimes)$ be the Carnot group of the elements in the tensor algebra of degree at most $[p]$ satisfying the shuffle product property, endowed with the homogeneous norm: $\|.\|:(1,g_1,\ldots,g_{[p]})\mapsto \max_{1\leq i \leq [p]}\|i!g_i\|^{1/i}$. It is interesting to see the analogy between the restriction of $\tilde{d}_p$ to the space of geometric rough paths and the $p$-variation norm for paths taking their values in $(G^{[p]},\otimes)$ (see \cite{LV} for details).
	\end{Rem}
	We can also define the concatenation of paths taking their values in the truncated tensor algebra:
	%%%%	
	%%%% DEF: Concatenation of Geometric rough paths 
	%%%%
	\begin{Def} Let $n\in\mathbb{N}^*$. Let $E$ be a vector space. Let $s,u,t\in\mathbb{R}$ such that $s\leq u \leq t$. Let $X$ (resp. $Y$) be a functional defined on $\Delta_{[s,u]}$ (resp. on $\Delta_{[u,t]}$) with values in $T^{n}(E)$. We define the concatenation of $X$ and $Y$, denoted $X*Y$, to be the functional over $\Delta_{[s,t]}$ defined as follows: for $(a,b)\in\Delta_{[s,t]}$
	\[(X*Y)_{(a,b)}=\left\{\begin{array}{lcl}
	X_{(a,b)}&,&\textrm{ if }b\leq u\\
	X_{(a,u)}\otimes Y_{(a,u)}&,&\textrm{ if }a\leq u\leq b\\
	Y_{(a,b)}&,&\textrm{ if }u\leq a\\
	\end{array}\right.
	\]
	\end{Def}
	The following theorem is straight-forward:
	\begin{theo} Let $p\geq 1$. Let $E$ be a Banach space. Let $s,u,t\in\mathbb{R}$ such that $s\leq u \leq t$. Let $X$ (resp. $Y$) be a functional defined on $\Delta_{[s,u]}$ (resp. on $\Delta_{[u,t]}$) with values in $T^{[p]}(E)$. Then:
	\begin{itemize}
	\item If $X$ and $Y$ are multiplicative functionals, then $X*Y$ is a multiplicative functional;
	\item If $X$ and $Y$ have finite $p$-variation, then $X*Y$ has finite $p$-variation;
	\item If $X$ and $Y$ are geometric $p$-rough paths, then $X*Y$ is a geometric $p$-rough path.
	\end{itemize}
	\end{theo}
	It will be useful to us to attach a starting point to our geometric rough paths as we will be mostly dealing with integrals of rough paths and rough paths on manifolds; both of which require a starting point. In the next few sections, a geometric rough path in a Banach space $E$ will be a pair $(x,X)$, where $X$ is a geometric rough path in the sense of definition \ref{DefGeoRP} and $x\in E$ is called the starting point. We will  identify the space of geometric rough paths with starting points with the Cartesian product $E \times G\Omega_p^{[0,T]}(E)$.\\
	%on which we define an equivalence relation $\sim$ that makes two rough paths with the same increments equivalent, i.e.:
	%\[(x,X)\sim (y,Y) \Leftrightarrow X=Y\]
	 On $E\times G\Omega_p^{[0,T]}(E)$ we define a metric $d_p$ as the product metric of $\tilde{d}_p$ and the norm on $E$, i.e.:
	\[d_p((x,X),(y,Y))=\max(\|x-y\|,\tilde{d}_p(X,Y))\]	
	In other situations, it will be more convenient to attach a trace to a rough path instead of a starting point (i.e. a path which increments correspond the element of degree 1 in said rough path).
	%%%
	%% DEF: Local Rough Paths
	%%%	
	\begin{Def}\label{DefLocalGRP}
	Let $E$ be a Banach space and $p\geq1$. For an open subset $U$ of $E$, a local geometric $p$-rough path in $U$ is a triple $(x,X,J)$ such that:
	\begin{itemize}
	\item $J$ is a compact interval.
	\item $x$ is a $U$-valued path over $J$.
	\item $X$ is an $E$-valued geometric $p$-rough path over $J$ which trace is $x$, i.e. for all $(s,t)\in\Delta_{J}$, $X^1_{s,t}=x_t-x_s$.
	\end{itemize}
	The set of local geometric $p$-rough paths in $U$ defined over compact sub-intervals of $I$ will be denoted $G\Omega_p^I(U;E)$, that of local geometric $p$-rough paths in $U$ defined over a compact interval $J$ will be denoted $G\Omega_p(J;U;E)$.
	\end{Def}
	\begin{Rem} The trace of a geometric $p$-rough path is trivially a path of finite $p$-variation.
	\end{Rem}
	\begin{Rem} It goes without saying that rough paths with starting points easily identify with local rough paths. Indeed, if $J$ is of the form, say, $[0,T]$, we will identify the local rough path $(x,X)$ (omitting the interval $J$ when no confusion is possible) with the rough path with starting point $(x(0),X)$.
	\end{Rem}
	%%
	% The integral of Lipschitz one-forms along geometric $p$-rough paths
	%%
	\subsubsection{The integral of Lipschitz one-forms along geometric $p$-rough paths}
	When trying to make sense of integrals of one-forms along rough paths, it is very important (for example in regard to solving a differential equation) to be able to control the smoothness (in terms of variation) of the image of the path under the one-form. It appears that Lipschitz maps, in the sense of Stein \cite{Stein}, are the appropriate type of maps to use in this context (for an extensive dedicated study of these maps, see for example \cite{Boutaib}):
	%%%%	
	%%%% DEFINITION LIPSCHITZ FUNCTIONS
	%%%%
	\begin{Def}\label{DefLipMap} Let $n\in \mathbb{N}$ and $0< \varepsilon\leq 1$. Let $E$ and $F$ be two normed vector spaces and $U$ be a subset of $E$. Let $f^0:U\to F$ be a map and for every $k \in [\![1,n]\!]$, let $f^k:U\to \mathcal{L}_s(E^{\otimes k},F)$ be a map with values in the space of the symmetric $k$-linear mappings from $E$ to $F$. \\
	For $k \in [\![0,n]\!]$, the map $R_k:E\times E\to \mathcal{L}(E^{\otimes k},F)$ defined by:
	\[\forall x,y \in U, \forall v \in E^{\otimes k}: f^k(x)(v)=\sum\limits_{j=k}^{n} f^j(y)(\frac{v\otimes (x-y)^{\otimes (j-k)}}{(j-k)!})+R_k(x,y)(v)\]
	is called the remainder of order $k$ associated to $f=(f^0,f^1,\ldots,f^n)$.\\	
	The collection $f=(f^0,f^1,\ldots,f^n)$ is said to be Lipschitz of degree $n+\varepsilon$ on $U$ (or in short a $\textrm{Lip}-(n+\varepsilon)$ map) if there exists a constant $M$ such that for all $k \in [\![0,n]\!]$, $x,y \in U$ and $v_1,\ldots,v_k \in E$:
	\begin{enumerate}
	\item $\|f^k(x)(v_1\otimes\cdots\otimes v_k)\| \leq M \|v_1\otimes\cdots\otimes v_k\| $;
	\item $\|R_k(x,y)(v_1\otimes\cdots\otimes v_k)\|\leq M \|x-y\|^{n+\varepsilon-k}\|v_1\otimes\cdots\otimes v_k\|$.
	\end{enumerate}
     The smallest constant $M$ for which the properties above hold is called the $\textrm{Lip}-(n+\varepsilon)$-norm of $f$ and is denoted by $\|f\|_{\textrm{Lip}-(n+\varepsilon)}$. The set of all $\textrm{Lip}-(n+\varepsilon)$ maps defined on $U$ with values in $F$ will be denoted $\textrm{Lip}(n+\varepsilon, U , F)$.
     \end{Def}
	%%%%
	%%%% Derivatives of a Lip map
	%%%%
	\begin{Rem}
	Let us stress that the above definition is purely quantitative and makes sense even on discrete sets. On any non-empty open subset of $U$, $f^1,\ldots,f^n$ are the successive derivatives of $f^0$. However, these maps are not necessarily uniquely determined by $f^0$ on an arbitrary set $U$. Keeping this in mind, if $f^0:U\to F$ is a map such that there exist $f^1,\ldots,f^n$ such that $(f^0,f^1,\ldots,f^n)$ is $\textrm{Lip}-(n+\varepsilon)$, we will often say that $f^0$ is $\textrm{Lip}-(n+\varepsilon)$ with no mention of $f^1,\ldots,f^n$.
     \end{Rem}
	%%%%	
	%%%% THEO: Classical integration map for rough paths
	%%%%
	\begin{theo}\label{OneFormRP} \cite{CLL,Lyons} Let $\gamma,p \in\mathbb{R}$ such that $\gamma>p\geq1$. Let $E$ and $F$ be two Banach spaces. Let $\alpha:E \to \mathcal{L}_c(E,F)$ be a $\textrm{Lip}-(\gamma-1)$ one-form. There exists a unique continuous map:
	\[I_{\alpha}:(E\times G\Omega_p^{[0,T]}(E),d_p)\longrightarrow(G\Omega_p^{[0,T]}(F),\tilde{d}_p)\]
	that extends Young's integration theory, i.e. for all $x\in\mathcal{V}^1([0,T],E)$, $I_{\alpha}(x_0,S_{[p]}(x))=S_{[p]}(\int \alpha(x)\mathrm{d}x)$. For a geometric $p$-rough path with a starting point $(x_0,X)$ or its corresponding local rough path $(x,X)$, we denote: 
	\[I_{\alpha}(x_0,X)=\int \alpha(x_0,X)\mathrm{d}X=\int \alpha(x,X)\mathrm{d}X\]
	\end{theo}
	%%%%%%%%%%%%%%%%%%%%%%%%%%%%%%%%%%%%%%%%%%%%%%%%%%%%%%%%
	%%%%%%%%%%%%%%%%%%%%%%%%%%%%%%%%%%%%%%%%%%%%%%%%%%%%%%%%
	%%
	% The Lipschitz geometry
	%%
	%%%%%%%%%%%%%%%%%%%%%%%%%%%%%%%%%%%%%%%%%%%%%%%%%%%%%%%%
	%%%%%%%%%%%%%%%%%%%%%%%%%%%%%%%%%%%%%%%%%%%%%%%%%%%%%%%%
	\section{The Lipschitz geometry}
	In this section, we review two of the findings of \cite{CLLy} that are most relevant to our work in this paper: the Lipschitz structure and geometric rough paths on manifolds.
	%%%%%%%%%%%
	%%
	% Lipschitz structures
	%%
	%%%%%%%%%%%
	\subsection{Lipschitz structures}
	We first introduce the general definitions of Lipschitz manifolds and Lipschitz maps and one-forms on them.
	%%%%
	%%%% DEF: Lipschitz manifold
	%%%%
	\begin{Def}\cite{CLLy} Let $\gamma>0$. Let $n\in\mathbb{N}^*$ and let $M$ be a $n$-topological manifold. Let $I$ be a countable set and, for every $i\in I$, $U_i$ be an open subset of $M$ and $\phi_i:M\to\mathbb{R}^n$ be a compactly supported map such that its restriction on $U_i$ defines a homeomorphism. We say that $((\phi_i,U_i))_{i\in I}$ is a Lipschitz-${\gamma}$ atlas if the following properties are satisfied:
	\begin{itemize}\item $(U_i)_{i\in I}$ is a pre-compact locally finite cover of $M$;
	\item For every $i\in I$: $\phi_i(U_i)=B(0,1)$.
	\item There exists $\delta\in(0,1)$, such that $(U_i^{\delta})_{i\in I}$ covers $M$, where, for every $i\in I$: $U_i^{\delta}={\phi^{-1}_i}_{|U_i}(B(0,1-\delta))$;
	\item There exists $L>0$, such that, for every $i,j\in I$, $\phi_j\circ{({\phi_i}_{|U_i})}^{-1}:B(0,1)\to\mathbb{R}^n$ is Lipschitz-$\gamma$ and $\|\phi_j\circ{({\phi_i}_{|U_i})}^{-1}\|_{\textrm{Lip}-\gamma}\leq L$.
	\end{itemize}
	With the constants above, we will say that $M$ is a Lipschitz-${\gamma}$ manifold with constants $(\delta,L)$.
	\end{Def}
	%%%%
	%%%% EX: Vector spaces are manifolds
	%%%%
	\begin{example} As one would expect, finite-dimensional vector spaces can be endowed with a natural Lipschitz-$\gamma$ manifold structure of any degree $\gamma\geq1$.\end{example}
	\begin{proof} Indeed, let $\gamma\geq1$. Let $V$ be a finite dimensional space and let $(e_1,\ldots,e_n)$ be a basis for $V$. Let $\varphi$ be a Lip-$\gamma$ extension on $V$ of $Id_{B(0,1)}$ with support in $B(0,2)$. For $x\in V$, let $\varphi_x$ be the map: $y\mapsto \varphi(y-x)$. Then $(\varphi_x,B(x,1))_{x\in I}$ is a Lip-$\gamma$ atlas on $V$, where:
	\[I=\left\{\sum_{i=1}^{n}\frac{k_i}{2} e_i|\, k_1,\ldots,k_n \in \mathbb{Z}\right\}\]
	\end{proof}
	%%%%
	%%%% EX: Compact manifolds are Lipschitz
	%%%%
	\begin{example}\cite{CLLy}\footnote{This is a slightly stronger result than in \cite{CLLy} but its proof is practically the same.} Let $n\in\mathbb{N}^*$ and $0<\varepsilon\leq 1$. A $\mathcal{C}^{n+1}$ structure on a compact manifold induces a natural Lip-$(n+\varepsilon)$ structure.\end{example}
	%%%%
	%%%% DEF: Lip maps and one forms on manifolds
	%%%%
	\begin{Def}\label{DefinitionLipFuncManifold} Let $\gamma_0\geq 1$ and $d\in \mathbb{N}^*$. Let $M$ be a $d$-dimensional $\textrm{Lip}-\gamma_0$ manifold with an atlas $\{(\phi_i,U_i), i\in I\}$ and $E$ be a normed vector space. 
	\begin{itemize}
		\item A map $f:M\to E$ is said to be $\textrm{Lip}-\gamma$, for $\gamma\leq \gamma_0$, if there exists a constant $C$ such that, for every $i\in I$, $f\circ{\phi_i}_{|U_i}^{-1}:B(0,1)\to E$ is $\textrm{Lip}-\gamma$ with a Lipschitz norm at most $C$. The smallest constant $C$ for which this property holds is called the $\textrm{Lip}-\gamma$ norm of $f$ and is denoted by $\|f\|_{\textrm{Lip}-\gamma}$.
		\item An $E$-valued one-form $\alpha$ on $M$ is said to be $\textrm{Lip}-\gamma$, for $\gamma\leq \gamma_0-1$, if there exists a constant $C$ such that, for every $i\in I$:
	\[({\phi_i}_{|U_i}^{-1})^*\alpha:B(0,1)\to \mathcal{L}(\mathbb{R}^d,E)\] 
	is $\textrm{Lip}-\gamma$ with a Lipschitz norm at most $C$. The smallest constant $C$ for which this property holds is called the $\textrm{Lip}-\gamma$ norm of $\alpha$ and is denoted by $\|\alpha\|_{\textrm{Lip}-\gamma}$.
	\end{itemize}
	\end{Def}
	\begin{Rem}
	An important property that is omitted in \cite{CLLy} and that we underline in the above definition is that, on a $\textrm{Lip}-\gamma_0$ manifold, it does not make sense geometrically to define $\textrm{Lip}-\gamma$ maps or $\textrm{Lip}-(\gamma-1)$ one-forms if $\gamma>\gamma_0$. Indeed, with the notations of definition \ref{DefinitionLipFuncManifold}, if for $i\in I$, $f\circ{\phi_i}_{|U_i}^{-1}:B(0,1)\to E$ is $\textrm{Lip}-\gamma$ (where $\gamma>\gamma_0$), then, based only on the definition of a $\textrm{Lip}-\gamma_0$ atlas, one cannot show that, for $j\in I$, the restriction of $f\circ{\phi_j}_{|U_j\cap U_i}^{-1}:B(0,1)\to E$ on any non-trivial subset of $U_j\cap U_i$ is smoother than $\textrm{Lip}-\gamma_0$. The same principle applies to defining Lipschitz one-forms. A more rigorous way to state the above is obtained by translating it in the language of equivalent Lipschitz structures (see \cite{CLLy} for a definition).
	\end{Rem}
	%%%%%%%%%%%
	%%
	% Rough paths on a manifold
	%%
	%%%%%%%%%%%
	\subsection{Rough paths on a manifold}
	In this subsection, we generalize the notion of rough paths to manifolds (as presented in \cite{CLLy}). As we don't have a natural notion of linearity and iterated integrals on a manifold, we have to consider a different approach than the one arising from the $p$-variational properties of paths and signatures. As we will hint to later, integrals of Lipschitz one-forms along rough paths do characterize the path; this is the first direction that we will take to define our rough paths. Additionally, in the absence of a natural translation, a rough path on a manifold comes attached with a starting point.
	%%%%
	%%%% DEF: Rough paths on manifolds
	%%%%
	\begin{Def}\cite{CLLy} Let $\gamma_0, p\in \mathbb{R}$ such that $\gamma_0>p\geq1$ and $T\geq0$. Let $M$ be a $\textrm{Lip}-\gamma_0$ manifold and $x\in M$. $\mathbb{X}$ is a geometric $p$-rough path over $[0,T]$ on $M$ starting at $x$ if, for every $\gamma\in \mathbb{R}$ such that $\gamma_0\geq\gamma>p$ and every Banach space $E$, the following conditions are satisfied:\begin{enumerate}
	\item $\mathbb{X}$ maps $\textrm{Lip}-(\gamma-1)$ $E$-valued one-forms on $M$ to $E$-valued geometric $p$-rough paths (in the classical sense).
	\item There exists a control $\omega$ such that for every $E$-valued $\textrm{Lip}-(\gamma-1)$ one-form $\alpha$ on $M$, $\mathbb{X}(\alpha)$ is controlled by $\|\alpha\|_{\textrm{Lip}-(\gamma-1)}\omega$, i.e.:
	\[\forall (s,t)\in \Delta_{[0,T]}, \forall i\in [\![1,[p]]\!]:\quad \|\mathbb{X}(\alpha)_{(s,t)}^i\|\leq\frac{(\|\alpha\|_{\textrm{Lip}-(\gamma-1)}\omega(s,t))^{i/p}}{\beta_p (i/p)!}\]
	\item(Chain rule) For every Banach space $F$ and every compactly supported $\textrm{Lip}-\gamma$ map $\psi: M\to F$ and every $E$-valued $\textrm{Lip}-(\gamma-1)$ one-form $\alpha$ on $F$ we have:
	\[\mathbb{X}(\psi^*\alpha)=\int{\alpha (\psi(x),\mathbb{X}(\psi_*)) \mathrm{d}\mathbb{X}(\psi_*)}\]
	\end{enumerate}
	\end{Def}
	Contrary to the classical framework, in the context of manifolds, we do not need to make a difference between rough paths and geometric rough paths (as only the latter are defined). Consequently, we drop the word ``geometric'' when talking about geometric rough paths on manifolds. Moreover, in the classical sense, geometric rough paths are uniquely determined by the values of the integrals of compactly supported one-forms along them. In order to make the correspondance one-to-one between the concepts of classical geometric rough paths (on finite-dimensional spaces) and rough paths on the same space when endowed with its canonical Lipschitz structure, we define the following equivalence relation:
	%%%%
	%%%% DEF: Equivalence relation on rough paths on manifolds
	%%%%
	\begin{Def}\label{EquiRelationManRP}\cite{CLLy} Let $\gamma_0, p\in \mathbb{R}$ such that $\gamma_0>p\geq1$. Let $M$ be a $\textrm{Lip}-\gamma_0$ manifold. We say that two $p$-rough paths $\mathbb{X}$ and $\tilde{\mathbb{X}}$ on $M$ are equivalent, and we write $\mathbb{X}\sim\tilde{\mathbb{X}}$, if they have the same starting point and if, for every $\gamma\in\mathbb{R}$ such that $\gamma_0\geq\gamma>p$ and for every $\textrm{Lip}-(\gamma-1)$ compactly supported Banach space valued one-form $\alpha$ on $M$, we have $\mathbb{X}(\alpha)=\tilde{\mathbb{X}}(\alpha)$. 
	\end{Def}
	We give now a hint on how to build a one-to-one correspondance between rough paths in the classical sense and in a Lip-$\gamma$ manifold, when said manifold is a finite-dimensional vector space $V$ (see \cite{CLLy} for the complete proof):
	\begin{itemize}
	\item Given a geometric $p$-rough path $(x,X)$ on $V$, we define the rough path $\mathbb{X}$ in the manifold by the starting point $x$ and the functional sending every compactly supported Banach space-valued Lip-$(\gamma-1)$ one form $\alpha$ to the classical rough path:
	\[\mathbb{X}(\alpha)=\int \alpha(x,X)\mathrm{d}X\]
	\item Conversely, given a rough path $\mathbb{X}$ on $V$ in the manifold sense, we are tempted to retrieve a rough path in the classical sense by integrating $\mathrm{d}Id_{V}$ along $\mathbb{X}$. Since $\mathrm{d}Id_{V}$ is not compactly supported (this is important for the definition to be consistent across the equivalence classes of the equivalence relation introduced in definition \ref{EquiRelationManRP}) and that $Id_{V}$ is not Lip-$\gamma$, we can intuitively replace $Id_{V}$ with a compactly supported Lip-$\gamma$ extension of its restriction on a set containing the ``support" of $\mathbb{X}$. The notion of support is not yet defined at this point but we can still have a good guess at a set containing it by using the control of the $p$-variation of $\mathbb{X}$.
	\end{itemize}
	When working on a manifold, one has to ensure that certain properties are invariant by the change of charts (or, in other words, by local reparametrisation). For this reason, we define the pushforward of rough paths by conveniently chosen maps:
	%%%%
	%%%% LEMMA: Pushforwards of rough paths
	%%%%
	\begin{lemma}\cite{CLLy} Let $\gamma_0, p \in \mathbb{R}$ such that $\gamma_0>p\geq1$. Let $M$ and $N$ be $\textrm{Lip}-\gamma_0$ manifolds and $f:M\to N$ be a map such that there exists a constant $C_f$ such that, for all $\gamma\in(p,\gamma_0]$ and every $\textrm{Lip}-(\gamma-1)$ Banach space valued one-form $\alpha$ on $M$, we have:
	\[\|f^*\alpha\|_{\textrm{Lip}-(\gamma-1)}\leq C_f\|\alpha\|_{\textrm{Lip}-(\gamma-1)}\]
	Then $f$ induces a pushforward map $f_*$ from $p$-rough paths on $M$ to $p$-rough paths on $N$ defined as follows: for every $p$-rough path $\mathbb{X}$ over $[0,T]$ on $M$ starting at $x$, $f_*\mathbb{X}$ starts at $f(x)$ and for every $\textrm{Lip}-(\gamma-1)$ Banach space valued one form $\alpha$ on $M$, where $\gamma\in(p,\gamma_0]$, $f_*\mathbb{X}(\alpha)$ is given by $f_*\mathbb{X}(\alpha)=\mathbb{X}(f^*\alpha)$.
	\end{lemma}
	The next proposition shows that there exists a particular class of Lipschitz maps that induce pushforwards of rough paths. In particular, we can ascertain that the pushforwards of rough paths by coordinate maps or transition maps are also rough paths:
	%%%%
	%%%% Prop: Control of pullbacks of one-forms
	%%%%
	\begin{Prop}\label{WellControlPshFwd}\cite{CLLy}\footnote{Following \cite{Boutaib}, we correct the exponent in the inequality compared to the result that appeared in \cite{CLLy} and drop furthermore the dependence on the dimension of the manifold. Said inequality can also be made sharper using improved estimates in \cite{Boutaib}.} Let $\gamma_0\geq \gamma>1$ and let $M$ be a Lip-$\gamma_0$ manifold and $W$ be a normed vector space. Let $f:M\to W$ be a $\textrm{Lip}-\gamma$ map and $\alpha$ be a $\textrm{Lip}-(\gamma-1)$ Banach space valued one-form on $W$ (or defined on a subset of $W$ containing $f(M)$). Then $f^*\alpha$ is Lip-$(\gamma-1)$ and there exists a constant $C_\gamma$ depending only on $\gamma$ such that:
	\[\|f^*\alpha\|_{\textrm{Lip}-(\gamma-1)}\leq C_\gamma\|\alpha\|_{\textrm{Lip}-(\gamma-1)}\|f\|_{\textrm{Lip}-\gamma}\max(\|f\|^{\gamma-1}_{\textrm{Lip}-\gamma},1)\]
	\end{Prop}
	Like in the classical case, we can define the concatenation of two rough paths. In the context of manifolds, this is important on its own since one usually works locally on coordinate domains to solve ordinary or rough differential equations before attempting to make sense of a global solution via a concatenation procedure:
	%%%%
	%%%% DEF: Concatenation of Manifold RP
	%%%%
	\begin{Def}\cite{CLLy} Let $\gamma_0> p\geq 1$ and $s\leq u\leq t$. Let $M$ be a $\textrm{Lip}-\gamma_0$ manifold. Let $\mathbb{X}$ (respectively $\mathbb{Y}$) be a $p$-rough path on $M$ over $[s,u]$ (resp. $[u,t]$). We define the concatenation of $\mathbb{X}$ and $\mathbb{Y}$, denoted by $\mathbb{X}*\mathbb{Y}$, to be the functional over $[s,t]$ mapping every Banach space-valued Lip-$(\gamma-1)$ compactly supported one-form $\alpha$ (for every $\gamma_0\geq \gamma> p$) to the classical rough path $\mathbb{X}(\alpha)*\mathbb{Y}(\alpha)$.
	\end{Def}
	Unlike the classical case, the concatenation of two rough paths on a manifold is not necessarily a rough path. This is due to the fact that rough paths on a manifold come attached with a starting point and that we have no natural notion of translation. Therefore, for this concatenation to be a rough path, we have to make sure that the two rough paths in question have starting and ``ending" points that agree in the following sense:
	%%%
	%% Def: endpoint=hom(startpoint)
	%%%
	\begin{Def} Let $\gamma_0> p\geq 1$ and $s\leq t$. Let $M$ be a $\textrm{Lip}-\gamma_0$ manifold. Let $\mathbb{X}$ be a $p$-rough path on $M$ over $[s,t]$ with starting point $x$ and let $y\in M$. We say that $\mathbb{X}$ has an end point consistent with $y$ if for every Banach space-valued compactly supported Lip-$\gamma$ map $f$ on $M$ (for every $\gamma_0\geq \gamma> p$), we have:
	\[f(x)+\mathbb{X}(f_*)^1_{s,t}=f(y)\]
	\end{Def}
	In this case, we can check the consistency condition for the concatenation of two rough paths and prove that it is also a rough path:
	%%%
	%% Prop: Concat RP=RP
	%%%
	\begin{Prop}\cite{CLLy} Let $\gamma_0> p\geq 1$ and $s\leq u\leq t$. Let $M$ be a $\textrm{Lip}-\gamma_0$ manifold. Let $\mathbb{X}$ (respectively $\mathbb{Y}$) be a $p$-rough path on $M$ over $[s,u]$ (resp. $[u,t]$) with starting point $x$ (resp. $y$). We assume that $\mathbb{X}$ has an end point consistent with the starting point of $\mathbb{Y}$. Then $\mathbb{X}*\mathbb{Y}$ defines a $p$-rough path on $M$ over $[s,t]$ with $x$ as a starting point.
	\end{Prop}
	Well-defined concatenations of rough paths constitute an associative operation:
	%%%
	%% Lemma: Concatenation is associative
	%%%
	\begin{lemma}\cite{CLLy} Let $\gamma_0> p\geq 1$ and $M$ be a $\textrm{Lip}-\gamma_0$ manifold. Let $\mathbb{X}$, $\mathbb{Y}$ and $\mathbb{Z}$ be $p$-rough paths on $M$ such that:\begin{itemize}
	\item $\mathbb{X}$ has an end point consistent with the starting point of $\mathbb{Y}$;
	\item $\mathbb{Y}$ has an end point consistent with the starting point of $\mathbb{Z}$.
	\end{itemize}
	Then:
	\begin{itemize}
	\item $\mathbb{X}$ has an end point consistent with the starting point of $\mathbb{Y}*\mathbb{Z}$;
	\item $\mathbb{X}*\mathbb{Y}$ has an end point consistent with the starting point of $\mathbb{Z}$;
	\item $\mathbb{X}*(\mathbb{Y}*\mathbb{Z})=(\mathbb{X}*\mathbb{Y})*\mathbb{Z}$
	\end{itemize}
	In this case, we denote $\mathbb{X}*\mathbb{Y}*\mathbb{Z}:=\mathbb{X}*(\mathbb{Y}*\mathbb{Z})$.
	\end{lemma}
	Since this notion of rough paths does not attach, for the moment, an underlying path on the manifold to a rough path, we define a notion of support based on the images of one-forms supported in all possible open sets:
	%%%
	%% Def: support of rough path
	%%%
	\begin{Def}\cite{CLLy} Let $\gamma_0> p\geq 1$ and $M$ be a $\textrm{Lip}-\gamma_0$ manifold. Let $\mathbb{X}$ be a $p$-rough path on $M$ with starting point $x$. 
	\begin{enumerate}
	\item For an open subset $U$ of $M$, we say that $\mathbb{X}$ misses $U$ if, for every $\gamma_0\geq \gamma> p$ and every Banach space-valued compactly supported Lip-$(\gamma-1)$ one-form $\alpha$ on $M$ with support in $U$, we have $\mathbb{X}(\alpha)=0$.
	\item We define the support of the rough path $\mathbb{X}$ as the closed set:
	\[\mathrm{supp}(\mathbb{X}):=\{x\}\bigcup \left(M-\bigcup_{\mathbb{X} \textrm{ misses } U}U\right)\]
	\end{enumerate}
	\end{Def}	
	Naturally, this notion of support is consistent with the definition of support for a classical rough path:
	%%%
	%% Prop: support of a classical rough path
	%%%
	\begin{Prop}\cite{CLLy} Let $\gamma_0, p \in \mathbb{R}$ such that $\gamma_0>p\geq1$ and $T>0$. Let $V$ be a finite dimensional vector space endowed with its canonical structure of a Lip-$\gamma_0$ manifold. Let $(x,X)$ be a geometric $p$-rough path on $V$ over $[0,T]$ (in the classical sense), and let $\mathbb{X}$ be the $p$-rough path associated to it in the manifold sense, i.e. for every Banach space-valued $\textrm{Lip}-(\gamma-1)$ compactly supported one-form $\alpha$ on $V$: $\mathbb{X}(\alpha)=\int\alpha(x,X)\mathrm{d}X$. Then:
	\[\textrm{supp}(\mathbb{X})=\left\{x+X^{1}_{0,t},\, t\in [0,T ]\right\}\]
	\end{Prop}
	Like in the classical case, the support of a rough path on a manifold can be shown to be compact:
	%%%
	%% THM: support is compact
	%%%
	\begin{theo}\label{SupportMRP}\cite{CLLy} Let $\gamma_0> p\geq 1$. Let $M$ be a $\textrm{Lip}-\gamma_0$ manifold and $\mathbb{X}$ be a $p$-rough path on $M$. Then the support of $\mathbb{X}$ is compact.
	\end{theo}
	The claim of the previous theorem can be shown by writing a rough path on a manifold as the concatenation of rough paths that have their support contained in the domain of only one chart at a time as described in the following theorem. This paper will generalise this decomposition of rough paths on manifolds into a collection of classical ``local'' rough paths that are compatible in a certain sense.
	%%%
	%% Theo: Existence of localising sequence
	%%%
	\begin{theo}\label{LocalSequenceMRP}\cite{CLLy} Let $\gamma_0> p\geq 1$ and $T\geq 0$. Let $M$ be a $\textrm{Lip}-\gamma_0$ manifold. Let $\mathbb{X}$ be a $p$-rough path on $M$ defined over an interval $[0,T]$. Then there exists a subdivision $D=(s_i)_{0\leq i \leq N}$ of $[0,T]$ and a collection $(\mathbb{X}^i,x_i, (\phi_i,U_i))_{1\leq i \leq N}$ such that, for every $i\in [\![1,N]\!]$:
	\begin{enumerate}
		\item $\mathbb{X}^i$ is a $p$-rough path on $M$ defined over $[s_{i-1},s_i]$ with starting point $x_i$.
		\item (if $i<N$) $x_{i+1}$ is consistent with the endpoint of $\mathbb{X}^i$.
		\item $(\phi_i,U_i)$ is a Lip-$\gamma_0$ chart on $M$ such that
		\[\textrm{supp}(\mathbb{X}^i)\subseteq U_i\]
		and such that for every Banach space-valued Lip-$(\gamma-1)$ one-form $\alpha$ defined on $M$ (where $\gamma_0\geq \gamma > p$) and compactly supported Lip-$(\gamma-1)$ one-form $\xi_{\alpha}$ defined on $\mathbb{R}^d$ which agrees with $({\phi_i}_{|U_i}^{-1})^*(\alpha)$ on $B(0,1)$, we have:
		\[\mathbb{X}^i(\alpha)=((\phi_i)_*\mathbb{X})_{[s_{i-1},s_i]}(\xi_{\alpha})\]
		\item $\mathbb{X}=\mathbb{X}^1*\cdots*\mathbb{X}^N$.
	\end{enumerate}
	$(\mathbb{X}^i,x_i, (\phi_i,U_i))_{1\leq i \leq N}$ is called a localising sequence for $\mathbb{X}$.
	\end{theo}
	%%%%%
	%%%%%
	%% Local rough paths
	%%%%%
	%%%%%
	\section{Local rough paths}
	%%%%%
	%% Intervals
	%%%%%
	\subsection{Intervals}
	We briefly present here some basic topological properties of covers of intervals which will be of use when studying rough paths locally.
	%%%
	%% Def: locally finite + compact cover
	%%%
	%\begin{Def} Let $J$ be a topological space and $(K_i)_{i\in I}$ be a collection of subsets of $J$.
	%	\begin{enumerate}
	%	\item $(K_i)_{i\in I}$ is said to be locally finite if for each $x\in J$ there exists a neighbourhood $U_{x}$ of $x$ such that at most finitely many $K_i$'s have non-empty intersection with $U_{x}$.
	%	\item $(K_i)_{i\in I}$ is said to cover $J$ (or is a cover of $J$) if $J=\cup _{i\in I}K_i$.
	%	\item $(K_i)_{i\in I}$ is called a compact cover of $J$ if it is locally finite, covers $J$ and for $i\in I$, $K_i$ is compact.
	%	\end{enumerate}
	%\end{Def}
	%%%
	%% Def: locally finite + compact cover
	%%%
	\begin{Def} Let $M$ be a topological space. Let $J$ be a subset of $M$ and $(K_i)_{i\in I}$ be a collection of subsets of $M$.
		\begin{enumerate}
		\item $(K_i)_{i\in I}$ is said to be locally finite if for each $x\in M$ there exists a neighbourhood $U_{x}$ of $x$ such that at most finitely many $K_i$'s have non-empty intersection with $U_{x}$.
		\item $(K_i)_{i\in I}$ is said to cover $J$ (or is a cover of $J$) if $J\subseteq \cup _{i\in I}K_i$.
		\item $(K_i)_{i\in I}$ is called a compact cover of $J$ if it is locally finite, covers $J$ and for $i\in I$, $K_i$ is compact subset of $J$.
		\end{enumerate}
	\end{Def}
	The proof of the following lemma is straightforward:
	%%%
	%% Lemma: Compact covers for intervals
	%%%
	\begin{lemma}\label{ExistenceCover} Each interval $J$ of $\mathbb{R}$ admits a compact cover.
	\end{lemma}
	%%%
	%% Lemma: Finite Compact cover of a compact set
	%%%
	\begin{lemma}\label{FiniteCompactCover} Let $M$ be a topological space, $J$ be a compact subset of $M$ and $(K_i)_{i\in I}$ be a locally finite collection of compact sets that cover $J$ such that all the $K_i$'s intersect $J$. Then $I$ is finite.
	\end{lemma}
	\begin{proof}
	For every $x\in J$, let $U_x$ be an open neighbourhood of $x$ that intersects only finitely many of the $K_i$'s. Then $(U_x)_{x\in J}$ is an open cover of $J$. As $J$ is compact, there exists a finite subset $J_0$ of $J$ such that $\mathcal{J}_0=\{U_x; x\in J_0\}$ covers $J$. As every $K_i$ intersects $J$, and hence an element in the finite set $\mathcal{J}_0$, and since every element of $\mathcal{J}_0$ intersects only finitely many of the $K_i$'s, $I$ must be finite.
	\end{proof}
	%%%
	%% Cor: Compact covers are countable
	%%%
	\begin{Cor}\label{CptCoverCountable}
	Let $J$ be an interval and $(K_i)_{i\in I}$ be a compact cover for $J$. Then $I$ is countable. Moreover, if $J$ is compact, then $I$ is finite.
	\end{Cor}
	\begin{proof}
	Let $(J_n)_{n\in \mathbb{N}}$ be a non-decreasing sequence of compact intervals (in the sense of inclusion) such that $J=\cup_n J_n$. Let $n\in \mathbb{N}$ and define:
	\[I_n=\{i\in I;\quad K_i\cap J_n \neq \varnothing \}\]
	Then $(K_i)_{i\in I_n}$ is a locally finite collection of compact sets that cover $J_n$ such that all the $K_i$'s intersect $J_n$. By lemma \ref{FiniteCompactCover}, $I_n$ is finite. Note now that $(I_n)_{n\in \mathbb{N}}$ is a non-decreasing sequence of finite subsets of $I$. Moreover, $I=\cup_n I_n$. Hence, $I$ is countable.\\
	The case when $J$ is compact is covered by lemma \ref{FiniteCompactCover}.
	\end{proof}
	%%%
	%% Lemma: Cover of cover
	%%%
	\begin{lemma}\label{CoverofCover}
	If $(K_i)_{i\in I}$ is a compact cover of an interval $J$, and if for each $i\in I$, $(K_{j})_{j\in I_i}$ is a compact cover of $K_i$ then $\{K_{j}| j\in I_i, i\in I\}$ is also a compact cover of $J$.
	\end{lemma}
	\begin{proof}
	First note that:
	\[\cup_{i\in I, j\in I_i} K_{j}=\cup_{i\in I}(\cup_{j\in I_i} K_{j})=\cup_{i\in I}K_i=J\]
	The $K_{j}$'s ($j\in \cup_II_i$) are all compact subsets of $J$. Let $x\in J$. Let $\mathcal{V}_{x,J}$ be a neighbourhood of $x$ in $J$ that intersects finitely many $K_i$'s ($i\in I$). As $I_i$ is finite for every $i\in I$ (corollary \ref{CptCoverCountable}), then $\mathcal{V}_{x,J}$ intersects finitely many $K_{j}$'s ($j\in \cup_II_i$).\end{proof}
	%%%
	%% Def: Refinements
	%%%
	\begin{Def}
	Let $(U_i)_{i\in I}$ and $(V_j)_{j\in J}$ be two collections of sets. We say that $(V_j)_{j\in J}$ is a refinement of $(U_i)_{i\in I}$ if, for every $j\in J$, there exists $i\in I$ such that $V_j\subseteq U_i$.
	\end{Def}
	%%%
	%% Lemma: Compact refinement
	%%%
	\begin{lemma}\label{CompactRefining}
	Let $\mathcal{O}=(O_i)_{i\in I}$ be a cover of an interval $J$ by open sets and $\mathcal{K}=(K_h)_{h\in H}$ be a compact cover  for $J$. Then there exists a compact cover for $J$ that is a refinement of both $\mathcal{O}$ and $\mathcal{K}$.
	\end{lemma}
	\begin{proof}
	Let $h\in H$ and $x\in K_h$. As $x\in J$, then there exists $i_x\in I$ and $\alpha_x>0$ such that $(x-\alpha_x,x+\alpha_x)\cap J\subseteq O_{i_x}$.  As $K_h$ is compact and $((x-\alpha_x/2,x+\alpha_x/2))_{x\in K_h}$ covers $K_h$, then there exists a finite subset $P\subseteq K_h$ such that $([x-\alpha_x/2,x+\alpha_x/2])_{x\in P}$ covers $K_h$. Define $I_x=[x-\alpha_x/2,x+\alpha_x/2] \cap K_h$ for every $x\in P$. Then, for every $x\in P$, $I_x$ is a compact subset of $K_h$ that is contained in $O_{i_x}$. $(I_x)_{x\in P}$ is therefore a compact cover of $K_h$.  We conclude using lemma \ref{CoverofCover}.\end{proof}
	%%%
	%% Prop: segmentation by Lebesgue's number
	%%%
	\begin{Prop}\label{CompactSubdiv} Given any cover of a compact interval $J$ by open sets $(O_i)_{i\in I}$, there exists a (finite) subdivision $(a_j)_{0\leq j\leq n}$ of $J$ such that:
	\[\forall j\in[\![0,n-1]\!],\, \exists i\in I \textrm{ such that: } [a_j,a_{j+1}]\subseteq O_i\]
	\end{Prop}
	\begin{proof}
	By Lebesgue's number lemma, let $\delta>0$ such that for every subset $A$ of $J$ of diameter less than $\delta$ there exist $i\in I$, such that $A\subseteq O_i$. It suffices to take any subdivision of $J$ of mesh less than $\delta$ to conclude the proof.
	\end{proof}
	%%%
	%% Prop: segmentation into compacts
	%%%
	\begin{Prop}\label{CompactSubdivK} Given any cover of a compact interval $J$ by compact intervals $(K_i)_{i\in I}$, there exists a (finite) subdivision $(a_j)_{0\leq j\leq n}$ of $J$ such that:
	\[\forall j\in[\![0,n-1]\!],\, \exists i\in I \textrm{ such that: } [a_j,a_{j+1}]\subseteq K_i\]
	\end{Prop}
	\begin{proof}
	Without loss of generality, we may assume that $J=[0,1]$. For every $x\in [0,1)$, we claim that:
	\[\exists \varepsilon_x>0, \quad \exists i\in I: [x,x+\varepsilon_x]\subset K_i\]
	Indeed, assume the converse is true. Then there exists $x\in [0,1)$ such that:
	\[\forall n\in\mathbb{N}^*, \quad \forall i\in I, \quad \exists x_{i,n}\in \left[x,x+\frac{1}{n}\right]- K_i\]
	Define $I_0=\{i\in I, x\in K_i\}$. Then necessarily $I_0\subsetneq I$. Indeed, if $I_0=I$, we let $i_0$ be such that $1\in K_{i_0}$ and then, by convexity of $K_{i_0}$, we have $[x,1]\subset K_{i_0}$, which contradicts our assumption. Let $i^*\in I_0$. For all $n\in \mathbb{N}^*$, let $x_n$ be an element of $[x,x+\frac{1}{n}]- K_{i^*}$. Then, for all $n\in \mathbb{N}^*$ and $i\in I_0$, $x_n \neq x$ (since $x\in K_{i^*}$) and $x_n\notin K_i$ (by the same convexity argument used above). Hence, $(x_n)$ is a sequence in the compact set $\cup_{I-I_0} K_i$ converging to $x$, which leads to a contradiction. Therefore, the claim is true. For every $i\in I$, we write $K_i=[s_i,t_i]$. Let $i_0\in I$ such that $K_{i_0}$ contains a neighbourhood of $0$ with non-empty interior and for every $i\in I$ such that $t_i\neq 1$, let $r_i\in I$ be such that there exists $\varepsilon_i>0$ such that $[t_i,t_i+\varepsilon_i]\subset K_{r_i}$. We define $a_0=0$ and $a_1=t_{i_0}$. Then $a_0<a_1$ and $[a_0,a_1]\subset K_{i_0}$. We define the rest of the subdivision in a recursive way: for $q\geq 1$, given $a_q=t_{q^*}$, if $a_q=1$, then we are done, otherwise we set $a_{q+1}=t_{r_{q^*}}$ (we have then $a_q<a_{q+1}$ and $[a_q,a_{q+1}]\subset K_{r_{q^*}}$). It is clear that this procedure converges in a finite number of steps and produces the desired subdivision.
	\end{proof}
	%%%%%
	%% Locally Lipschitz maps
	%%%%%
	\subsection{Locally Lipschitz maps}
	%%%
	%% Def: Locally Lipschitz maps
	%%%
	\begin{Def}  Let $\gamma >0$. Let $E$ and $F$ be two normed vector spaces, $U$ be a subset of $E$ and $f:U\rightarrow F$ be a map. We say that $f$ is locally Lipschitz-$\gamma$ if, for every $x\in U$, there exists a neighborhood $\mathcal{V}_{x,U}$ of $x$ in $U$ such that $f_{|\mathcal{V}_{x,U}}$ is Lipschitz-$\gamma$. The set of all locally $\textrm{Lip}-\gamma$ maps defined on $U$ with values in $F$ will be denoted $\textrm{Lip}_{\textrm{loc}}(\gamma, U , F)$.
	\end{Def}
	%%%
	%% Trivial examples of Locally Lipschitz maps
	%%%
	\begin{example} Lipschitz-$\gamma$ maps are obviously locally Lipschitz-$\gamma$. Continuous linear and polynomial maps are locally Lipschitz (of any degree).
	\end{example}
	%%%
	%% Rem: Almost Lipschitz maps
	%%%
	\begin{Rem} One of the main reasons for introducing locally Lipschitz maps here instead of working with Lipschitz maps that are classical in the setting of geometric rough paths is to be able to use linear maps (such as the identity maps which are pivotal in the definition of categories) which are not Lipschitz in general\footnote{ a usual workaround in analysis is to use suitable Whitney extensions}. Another possible solution to this issue is the use of almost Lipschitz maps introduced in \cite{Boutaib}. The results below generalise automatically to this class of maps.
	\end{Rem}
	We recall here two important results on Lipschitz maps. They can be found for example in \cite{Boutaib} and \cite{CLLy}:
	%%%
	%% Theo: Composition of Lipschitz maps
	%%%
	\begin{theo}\label{Compolipfunc}\footnote{see \cite{Boutaib} for a slightly improved estimate} Let $E$, $F$ and $G$ be three normed vector spaces. Let $U$ be a subset of $E$ and $V$ be a subset of $F$. Let $\gamma\geq 1$. We assume that $(E^{\otimes k})_{k\geq1}$ and $(F^{\otimes k})_{k\geq1}$ are endowed with norms satisfying the projective property. Let $f:U \to F$ and $g:V \to G$ be two $\textrm{Lip}-\gamma$ maps such that $f(U) \subseteq V$. Then $g \circ f$ is $\textrm{Lip}-\gamma$ and there exists a constant $C_{\gamma}$ (depending only on $\gamma$) such that:
	\[\|g\circ f\|_{\textrm{Lip}-\gamma} \leq C_{\gamma}\|g\|_{\textrm{Lip}-\gamma}\max(\|f\|^{\gamma}_{\textrm{Lip}-\gamma},1)\]
	\end{theo}
	%%%
	%% Lemma: Lipschitzness and differentiability
	%%%
	\begin{lemma}\label{CaracLip} Let $n\in \mathbb{N}$, $0< \varepsilon\leq 1$ and $C\geq 0$. Let $E$ and $F$ be two normed vector spaces and $U$ be an open subset of $E$. Let $f:U\to F$ be a map and for every $k\in [\![1,n]\!]$, let $f^k:U\to \mathcal{L}(E^{\otimes k},F)$ be a map with values in the space of the symmetric $k$-linear mappings from $E$ to $F$. We consider the two following assertions:\begin{description}
	\item[(A1)] $(f,f^1,\ldots,f^n)$ is $\textrm{Lip}-(n+\varepsilon)$ and $\|f\|_{\textrm{Lip}-(n+\varepsilon)}\leq C$.
	\item[(A2)] $f$ is $n$ times differentiable, with $f^1,\ldots,f^n$ being its successive derivatives. $\|f\|_{\infty}$, $\|f^1\|_{\infty}$, $\ldots$, $\|f^n\|_{\infty}$ are upper-bounded by $C$ and for all $x,y\in U: \|f^n(x)-f^n(y)\|\leq C \|x-y\|^{\varepsilon}$.\end{description}
	Then $\mathbf{(A1)}\Rightarrow \mathbf{(A2)}$. If furthermore $U$ is convex then $\mathbf{(A1)}\Leftrightarrow \mathbf{(A2)}$.
	\end{lemma}
	Using for example lemma \ref{CaracLip}, one can easily show the following result:
	%%%
	%% Prop: Smooth maps are locally Lipschitz
	%%%
	\begin{Prop} Let $n\in \mathbb{N}^*$. Let $E$ and $F$ be two normed vector spaces, $U$ be an open subset of $E$ and $f:U\rightarrow F$ be a map of class $\mathcal{C}^n$. Then $f$ is locally Lip-$n$.
	\end{Prop}
	Unlike the notion of Lipschitzness, local Lipschitzness can easily be defined recursively on open sets:
	%%%
	%% Lemma: Local Lipschitzness is recursive
	%%%
	\begin{lemma} Let $\gamma >1$. Let $E$ and $F$ be two normed vector spaces, $U$ be an open subset of $E$ and $f:U\rightarrow F$ be a differentiable map. Then its derivative $\mathrm{d}f$ is locally Lipschitz-$(\gamma-1)$ if and only if $f$ is locally Lipschitz-$\gamma$.
	\end{lemma}
	\begin{proof} Notice that, on every ball $B(x,\alpha)\subseteq U$ on which $\mathrm{d}f$ is Lipschitz-$(\gamma-1)$, one can use for example the fundamental theorem of calculus to bound $f$ and we deduce that the restriction of $f$ on this set is Lipschitz-$\gamma$ using lemma \ref{CaracLip}. The converse is obvious.\end{proof}
	%%%
	%% Lemma: Composition of Local Lipschitz
	%%%	
	Following theorem \ref{Compolipfunc}, local Lipschitzness is conserved under composition:
	\begin{Prop} Let $\gamma\in [1,+\infty[$. Let $E$, $F$ and $G$ be normed vector spaces, $U$ be a subset of $E$ and $V$ be a subset of $F$. Let $f:U\rightarrow F$ and $g:V\rightarrow G$ be two locally Lipschitz-$\gamma$ maps such that $f(U)\subseteq V$. Then $g\circ f$ is locally Lipschitz-$\gamma$.
	\end{Prop}
	Before we carry on, we need the following embedding theorem for which the complete statement and proof can be found for example in \cite{Boutaib}.
	%%%
	%% THN: EMBEDDINGS OF LIP GENERAL
	%%%	
	\begin{theo}\label{EmbedLip2} Let $\gamma,\gamma'>0$ such that $\gamma'\leq \gamma$. Let $E$ and $F$ be two normed vector spaces and $U$ be a subset of $E$.  Let $f:U\to F$ be a $\textrm{Lip}-\gamma$ map. Then $f$ is $\textrm{Lip}-\gamma'$ and there exists a constant $M_{\gamma,\gamma'}$ (depending only on $\gamma$ and $\gamma'$) such that $\|f\|_{\textrm{Lip}-\gamma'}\leq M_{\gamma,\gamma'} \|f\|_{\textrm{Lip}-\gamma}$.
	\end{theo}
	Locally Lipschitz maps conserve the smoothness (in the sense of variation) of paths:	
	%%%
	%% THM: Image of p-var by local Lip
	%%%	
	\begin{theo}\label{LipImagepvar} Let $p,\gamma\in [1,+\infty[$. Let $E$ and $F$ be two normed vector spaces, $U$ be a subset of $E$ and $J$ a compact interval. Let $f:U\rightarrow F$ be a locally Lip-$\gamma$ map over $U$ and $x:J\rightarrow U$ a path with finite $p$-variation. Then $f\circ x:J\rightarrow F$ is of finite $p$-variation.
	\end{theo}
	\begin{proof} For every $t\in J$, let $B_t$ be an open neighborhood of $x_t$ in $U$ such that $f_{|B_t}$ is Lipschitz-$1$ (following theorem \ref{EmbedLip2}); denote its Lip-$1$ norm by $M_t$. Then $(x^{-1}(B_t))_{t\in J}$ is an open cover of $J$. Let $(a_i)_{0\leq i\leq n}$ be a subdivision of $J$ such that for all $i\in[\![0,n-1]\!]$, there exists $t_i\in J$ such that $[a_i,a_{i+1}]\subseteq x^{-1}(B_{t_i})$ (proposition \ref{CompactSubdiv}) and denote $M=\max_{1\leq i\leq n}M_{t_i}$.
	
	Let $\omega$ be a control of the $p$-variation of $x$ over $J$. Let $s,u\in J$ and let $q,r\in [\![0,n]\!]$ such that $a_q\leq s\leq\cdots\leq u\leq a_r$. Then we have, from the fact that $f$ is Lip-$1$ on each of the $B_{t_{i}}$'s, for $i\in [\![0,n-1]\!]$, with a norm less than $M$:
	\[\begin{array}{rcl}
	\|f(x_s)-f(x_u)\|&\leq&\|f(x_s)-f(x_{a_{q+1}})\|+\sum\limits_{k=q+1}^{r-2}\|f(x_{a_k})-f(x_{a_{k+1}})\|+\\
					&&\|f(x_{a_{r-1}})-f(x_u)\|\\
				&\leq&M\left(\|x_s-x_{a_{q+1}}\|+\sum\limits_{k=q+1}^{r-2} \|x_{a_k}-x_{a_{k+1}}\|+\|x_{a_{r-1}}-x_u\|\right)\\
				&\leq&M \left( \omega(s,a_{q+1})^{1/p}+\sum\limits_{k=q+1}^{r-2}\omega({a_k},a_{k+1})^{1/p}+\omega(a_{r-1},u)^{1/p}\right)\\
	\end{array}\]
	which, using the super-additivity of $\omega$ and Jensen's inequality, gives the control:
	\[\|f(x_s)-f(x_u)\|^p\leq M^pn^{p-1}\omega(s,u)\]
	Therefore, $f\circ x$ is of finite $p$-variation ($M$ and $n$ do not depend on $s$ or $u$).\end{proof}
	%%%%%
	%% A study of some properties of rough paths
	%%%%%
	\subsection{A study of some properties of rough paths}\label{sec:LocalRPBSpace}
	%%%
	%% Lemma: From multiplicative on compact to universal mult.
	%%%	
	\begin{lemma}\label{LocalMultiGlobalId} Let $E$ be a vector space and $J$ be a compact interval. Let $X$ and $Y$ be two $E$-valued multiplicative functionals on $J$. Let $(K_i)_{i\in I}$ be a compact cover for $J$ by compact intervals such that, for all $i\in I$, $X_{|K_i}$ and $ Y_{|K_i}$ are equal. Then $X$ and $ Y$ are equal on $J$.
	\end{lemma}
	\begin{proof} Let $(a,b)\in\Delta_J$. Then $(K_i\cap[a,b])_{i\in S}$, where $S=\{i\in I|K_i\cap[a,b]\neq\varnothing\}$, is a compact cover for $[a,b]$ by compact intervals. Let $(a_j)_{0\leq j \leq n}$ be a subdivision of $[a,b]$ such that for all $j\in [\![0,n-1]\!]$, there exists $i\in S$ such that $[a_j,a_{j+1}]\subseteq K_i$ (proposition \ref{CompactSubdivK}). Then, by assumption: $\forall j\in [\![0,n-1]\!]\; X_{a_j,a_{j+1}}=Y_{a_j,a_{j+1}}$. Therefore:
	\[X_{a_0,a_{1}}\otimes \cdots \otimes X_{a_{n-1},a_{n}}=Y_{a_0,a_{1}}\otimes \cdots \otimes Y_{a_{n-1},a_{n}}\]
	Which, by using the multiplicativity of $X$ and $Y$, gives: $X_{a,b}=Y_{a,b}$. Since this holds for all $(a,b)\in\Delta_J$, then $X=Y$.\end{proof}
	%%%
	%% Prop: From  compact to universal mult.
	%%%	
	\begin{lemma}\label{LocalMultiGlobalIdII} Let $E$ be a vector space and $J$ be a compact interval. Let $(K_i)_{i\in I}$ be a compact cover for $J$ by compact intervals. Let $(X_i)_{i\in I}$ be a collection of $E$-valued multiplicative functionals such that, for all $i\in I$, $X_i$ is defined over $K_i$ and for $i,j\in I$ such that $K_i\cap K_j\neq \varnothing$, we have ${X_i}_{|K_i\cap K_j}={X_j}_{|K_i\cap K_j}$. Then there exists a unique multiplicative functional $X$ defined over $J$ such that, for all $i\in I$, $X_{|K_i}=X_i$.
	\end{lemma}
	\begin{proof} Let $(a_j)_{0\leq j \leq n}$ be a subdivision of $J$ and $(i_j)_{0\leq j \leq n}$ be a finite sequence of elements of $I$ such that for all $j\in [\![0,n-1]\!]$, we have $[a_j,a_{j+1}]\subseteq K_{i_j}$. Let $(s,t)\in \Delta_J$ and let $q,r\in [\![1,n-1]\!]$ be such that:
	\[a_{q-1}\leq s\leq a_q\leq \cdots \leq a_r \leq t \leq a_{r+1}\]
	 We set:
	\[X_{s,t}=X_{i_{q-1}}(s,a_{q})\otimes \cdots \otimes X_{i_{r}}(a_{r},t)\]
	It is an easy exercise to check that $X$ defines a multiplicative functional satisfying the requirements of the statements. The uniqueness of such functional is a consequence of lemma \ref{LocalMultiGlobalId}.
	\end{proof}
	Before we make sure that the integral of a geometric rough path against a sufficiently smooth locally Lipschitz one-form is indeed well defined, we need the following lemma quantifying the distance over an interval in the variation topology between two rough paths given their distance on a subdivision of that interval.
	%%%
	%% Lemma: Comparing distance VERSION II
	%%%	
	\begin{lemma}\label{CprRPDistSubdiv} Let $E$ be a Banach space, $p\in [1,+\infty[$ and $T\geq 0$. Let $Y$ and $Z$ be two continuous maps from $\Delta_{[0,T]}$ to $T^{[ p]}(E)$ with finite $p$-variation. Given a subdivision $\mathcal{D}=(s_i)_{0\leq i\leq r}$ of $[0,T]$, then we have:
		\[\tilde{d}_p(Y,Z)\leq 3^{1-1/p}\max_{1\leq j\leq[p]}\|(\tilde{d}_p^{[0,s_{1}]}(Y,Z), \cdots, \tilde{d}_p^{[s_{r-1},T]}(Y,Z))\|_{l_j}\]
		where, for a finite sequence $x=(x_1,x_2,\ldots,x_n)$ and $j>0$, we define:
		\[\|x\|_{l_j}=\left(\sum_{i=1}^n |x_i|^j\right)^{1/j}\]
	\end{lemma}
	\begin{proof}
	Let $j\in[\![1,[p]]\!]$. For $(s,u)\in \Delta_{[0,T]}$, define $V_{s,u}=Y^j_{s,u}-Z^j_{s,u}$. Let $\Delta=(t_i)_{0\leq i \leq q}$ be a subdivision of $[0,T]$. Let $i\in[\![0,q-1]\!]$. Define the subdivision $\mathcal{D}\cap[t_i,t_{i+1}]$ of $[t_i,t_{i+1}]$ to be $(m_l^i):=(t_i, s_{\tilde{r}_i},\ldots, s_{\tilde{r}_i+n_i}, t_{i+1})$, where $\tilde{r}_i$ and $n_i$ are such that (if they exist) $s_{\tilde{r}_i-1}<t_i\leq s_{\tilde{r}_i}$ and $s_{\tilde{r}_i+n_i}\leq t_{i+1}<s_{\tilde{r}_i+n_i+1}$. Then we have:
	\[\begin{array}{rcl}
	\|V_{t_i,t_{i+1}}\|^{p/j}&\leq&\left(\|V_{t_i,s_{\tilde{r}_i}}\|
							+ \sum\limits_{l=\tilde{r}_i}^{\tilde{r}_i+n_i} \tilde{d}_p^{[s_l,s_{l+1}]}(Y,Z)^{j}
							+\|V_{s_{\tilde{r}_i+n_i},t_{i+1}}\|\right)^{p/j}\\
							&\leq& 3^{p/j-1}\left(\|V_{t_i,s_{\tilde{r}_i}}\|^{p/j}
							+ \left(\sum\limits_{l=\tilde{r}_i}^{\tilde{r}_i+n_i} \tilde{d}_p^{[s_l,s_{l+1}]}(Y,Z)^{j}\right)^{p/j}
							+\|V_{s_{\tilde{r}_i+n_i},t_{i+1}}\|^{p/j}\right)\\
	\end{array}\]
	Using the inequality $a^{\alpha}+b^{\alpha}\leq (a+b)^{\alpha}$, for $a,b\geq 0$ and $\alpha\geq 1$ after summing over all $i$'s, we get:
	\[\sum\limits_{\Delta}\|V_{t_i,t_{i+1}}\|^{p/j}
		\leq 3^{p/j-1}\left(\sum\limits_{i=0}^{r-1} \tilde{d}_p^{[s_i,s_{i+1}]}(Y,Z)^{j}\right)^{p/j}
		=3^{p/j-1}\|(\tilde{d}_p^{[0,s_{1}]}(Y,Z), \cdots, \tilde{d}_p^{[s_{r-1},T]}(Y,Z))\|_{l_j}^p\]
		which gives the result.
	\end{proof}
	We can now show that the integral of a locally Lipschitz one-form along geometric rough paths is, as expected, well defined and continuous when varying the path.
	%%%
	%% Thm: Existence and Continuity of local integral
	%%%	
	\begin{theo} \label{ContIntegral} Let $p,\gamma\in [1,+\infty[$ such that $\gamma>p$ and $T\geq 0$. Let $E$ and $F$ be two Banach spaces and $U$ be an open subset of $E$. There exists a unique map:
	\[\mathcal{J}:\textrm{Lip}_{\textrm{loc}}(\gamma-1,U, \mathcal{L}(E,F))\times G\Omega_p^{[0,T]}(U;E) \longrightarrow G\Omega_p^{[0,T]}(F)\]
	such that if $\alpha:U\rightarrow\mathcal{L}(E,F)$ is a Lip-$(\gamma-1)$ one-form and $(x,X)$ is a local geometric $p$-rough path in $U$ defined over an interval $[s,t]\subseteq [0,T]$ then:
	\[\mathcal{J}(\alpha,(x,X))=\int{\alpha(x,X)\textrm{d}X}\]
	Moreover, for a locally Lip-$(\gamma-1)$ one-form $\alpha:U\rightarrow\mathcal{L}(E,F)$ and a subinterval $[s,t]\subseteq [0,T]$, the map:
	\[\begin{array}{rrcl}
	\mathcal{J}_{\alpha}:&(G\Omega_p([s,t];U;E),d_p)&\longrightarrow&(G\Omega_p([s,t]; F),\tilde{d}_p)\\
	&(x,X)&\mapsto& \mathcal{J}(\alpha,(x,X))\\
	\end{array}\]
	is continuous.
	\end{theo}
	\begin{proof} Let $\alpha:U\rightarrow\mathcal{L}(E,F)$ be a locally Lip-$(\gamma-1)$ one-form and $(x,X)\in G\Omega_p^{[0,T]}(U)$ defined over a compact interval $J$. Let $s,t \in \Delta_J$. For $u\in [s,t]$, let $O_u$ be an open neighbourhood of $x_u$ in $U$ such that $\alpha_{|O_u}$ is Lip-$(\gamma-1)$. As $(x^{-1}(O_u))_{s\leq u \leq t}$ is an open covering of $[s,u]$, let $(s_i)_{0\leq i\leq n}$ be a subdivision of $[s,t]$ such that for all $i\in [\![0,n-1]\!]$, there exists $u_i\in [s,t]$ such that $[s_i,s_{i+1}]\subset x^{-1}(O_{u_i})$.\\
	Let $i\in [\![0,n-1]\!]$. The rough path $(x,X)_{|[s_i,s_{i+1}]}$ takes its values in $O_{u_i}$ on which $\alpha$ is Lip-$(\gamma-1)$. Therefore, the geometric $p$-rough path $\int{\alpha(x(s_i),X)\textrm{d}X}$ is well defined over $[s_i,s_{i+1}]$. By lemma \ref{LocalMultiGlobalIdII}, there exists a unique geometric $p$-rough path $\mathcal{J}(\alpha,(x,X))$ defined over $[s,u]$ which agrees with it. This proves the existence of the map $\mathcal{J}$. The uniqueness is shown following the same reasoning and decomposition.\\
	 Let $(x,X)\in G\Omega_p([0,T]; U; E)$. For $u\in [0,T]$, let $\eta_u>0$ be such that $B(x_u,\eta_u)\subseteq U$ and $\alpha_{|B(x_u,\eta_u)}$ is Lip-$(\gamma-1)$. As before, let $(s_i)_{0\leq i\leq n}$ be a subdivision of $[0,T]$ such that for all $i\in [\![0,n-1]\!]$, there exists $u_i\in [0,T]$ such that $x([s_i,s_{i+1}])\subset B(x_{u_i},\eta_{u_i}/2)$ and denote $\eta=\min_i \eta_{u_i}$. Let $\varepsilon >0$. Let $\beta>0$ be such that for all $i\in [\![0,n-1]\!]$ and $(z,Z)\in G\Omega_p([s_i,s_{i+1}]; B(x_{u_i},\eta_{u_i}); E)$:
	 \[ d_p((x(s_i),X_{|[s_i,s_{i+1}]}),(z(s_i),Z))\leq \beta \Rightarrow 
	 \tilde{d}_p^{[s_i,s_{i+1}]}\left(\int{\alpha(x(s_i),X_{|[s_i,s_{i+1}]})\textrm{d}X},\int{\alpha(z(s_i),Z)\textrm{d}Z}\right)\leq \varepsilon\] 
	 Let $(y,Y)\in G\Omega_p([0,T]; U; E)$ such that $d_p((x(0),X),(y(0),Y))\leq \min(\beta,\eta)/4$. For $u\in [0,T]$, we have then:
	 \[\|x(u)-y(u)\|\leq \|x(0)-y(0)\|+\|X^1_{0,u}-Y^1_{0,u}\| \leq \min(\beta,\eta)/2\]
	 Let $i\in [\![0,n-1]\!]$. For $u\in [s_i,s_{i+1}]$, both $x_u$ and $y_u$ are in $B(x_{u_i},\eta_{u_i})$, hence
	 \[\tilde{d}_p^{[s_i,s_{i+1}]}(\mathcal{J}_{\alpha}(x,X),\mathcal{J}_{\alpha}(y,Y))\leq  {\varepsilon}\]
	 By lemma \ref{CprRPDistSubdiv} , we deduce that:
	 \[\tilde{d}_p(\mathcal{J}_{\alpha}(x,X),\mathcal{J}_{\alpha}(y,Y))\leq 
	 3^{1-1/p}r {\varepsilon}\]
	 from which we deduce the continuity of $\mathcal{J}_{\alpha}$ at  $(x,X)$.
	 \end{proof}
	%%%
	%% lemma: integral of signature is a signature
	%%%	
	 \begin{lemma}\label{IntegralSig} Let $p,\gamma\in [1,+\infty[$ such that $\gamma>p$. Let $E$ and $F$ be two Banach spaces and $U$ be an open subset of $E$. and $J$ an interval. Let $f:U\rightarrow F$ be a locally Lip-$\gamma$ map over $U$ and $x:J\rightarrow U$ a path with bounded variation. Then:
	 \[(f(x),\int{\textrm{d}f(x,S_{[p]}(x))\textrm{d}S_{[p]}(x)},J)=(f(x),S_{[p]}(f(x)),J)\]
	\end{lemma}	
	\begin{proof} First notice that $S_{[p]}(f(x))$ is well-defined since $f(x)$ has bounded variation by theorem \ref{LipImagepvar}. Let $s,t\in \Delta_J$:
	\begin{enumerate}
	\item As $\mathrm{d}f$ is continuous and $x$ is of bounded variation, then $\int{\mathrm{d}f(x,S_{[p]}(x))\textrm{d}S_{[p]}(x)}$ has finite 1-variation and is equal to (the signature of) the Stieltjes integral $\int{\textrm{d}f(x)\textrm{d}x}$. Therefore:
	\[\left(\int{\mathrm{d}f(x,S_{[p]}(x))\textrm{d}S_{[p]}(x)}\right)^1_{u,v}=f(x_v)-f(x_u)\]
	for all $(u,v)\in\Delta_{[s,t]}$.
	\item $S_{[p]}(f(x))$ has finite 1-variation and $S_{[p]}(f(x))^1_{u,v}=f(x_v)-f(x_u)$  for all $(u,v)\in\Delta_{[s,t]}$.
	\end{enumerate}
	Two multiplicative functionals that have finite 1-variation and which terms of the $1^{st}$ degree agree are equal by theorem \ref{ExtThRP}. Therefore, the sought identity stands.\end{proof}

	%%%%%
	%%%%%
	%% Rough paths and categories
	%%%%%
	%%%%%
	\section{Rough paths and categories}
	%%%%%
	%% Elements from category theory
	%%%%%
	\subsection{Elements from category theory}
	To highlight the minimal framework on which we can define the notion of rough paths, we will be using the language of categories.	We refer for example to \cite{AHS} for the definitions below which may also be gathered in \cite{Lee}.
	%%%
	%% DEF: Category
	%%%	
	\begin{Def}[Category] A category $\mathcal{C}$ is a triple $((U_i)_{i\in I},(\hom (U_i,U_j))_{i,j\in I},(\psi_{i,j,k})_{i,j,k\in I})$ satisfying the three following axioms:
	\begin{enumerate}
	\item For every $i,j \in I$, $\hom (U_i,U_j)$ is a set.
	\item For every $i,j,k\in I$, $\psi_{i,j,k}$ is a mapping from $\hom (U_i,U_j)\times \hom (U_j,U_k)$ into $\hom (U_i,U_k)$.
	\item\textbf{(Associativity)} For every $i,j,k,l\in I$, $f\in\hom (U_i,U_j)$, $g\in\hom (U_j,U_k)$ and $h\in\hom (U_k,U_l)$:
	\[\psi_{i,k,l}(\psi_{i,j,k}(f,g),h)=\psi_{i,j,l}(f,\psi_{j,k,l}(g,h))\]
	\item\textbf{(Existence of identities)} For every $i\in I$, there exists $\textrm{id}_{U_i}\in\hom (U_i,U_i)$ such that, for all $j\in I$, $g\in\hom (U_i,U_j)$ and $h\in\hom (U_j,U_i)$, we have:
	\[\psi_{i,i,j}(\textrm{id}_{U_i},g)=g \quad \textrm{and} \quad \psi_{j,i,i}(h,\textrm{id}_{U_i})=h\]
	\end{enumerate}
	In this case, for all $i,j,k\in I$, $U_i$ is called an object of $\mathcal{C}$; an element in $\hom (U_i,U_j)$ (also denoted by $\hom_{\mathcal{C}} (U_i,U_j)$) is called either an arrow, a morphism or a homomorphism in $\mathcal{C}$. The collections $(U_i)_{i\in I}$ and $(\hom_{\mathcal{C}} (U_i,U_j))_{i,j\in I}$ are respectively denoted $\mathrm{ob}(\mathcal{C})$ and $\hom (\mathcal{C})$. The binary operation $\psi_{i,j,k}$ is called a composition of morphisms and is simply denoted by $\circ$. i.e. for $i,j,k\in I$, $f\in\hom (U_i,U_j)$ and  $g\in\hom (U_j,U_k)$, $\psi_{i,j,k}(f,g)$ is denoted $g\circ f$.
	\end{Def}
	\begin{Rem} If $\mathcal{C}=((U_i)_{i\in I},(\hom_{\mathcal{C}} (U_i,U_j))_{i,j\in I},\circ)$ is a category, and with the notations of the previous definition, the associativity and the existence of identities axioms can be rewritten in the following way:
	\[h\circ (g \circ f)=(h\circ g) \circ f\]
	and
	\[g \circ\textrm{id}_{U_i}= g \quad \textrm{and} \quad \textrm{id}_{U_i}\circ h=h\]
	\end{Rem}
	%%%
	%% Examples of categories
	%%%		
	\begin{examples}
	\begin{itemize}
		\item By taking a family of sets considered as objects and all maps between these sets considered as arrows and the composition of maps as binary operations we obtain a category, usually called the category of sets.
		\item Let $(G_i)_{i\in I}$ be a family of groups. For every $i,j\in I$, let $\hom (G_i,G_j)$ be the set of all group homomorphisms from $G_i$ to $G_j$. Then $\mathcal{C}=((G_i)_{i\in I},(\hom (G_i,G_j))_{i,j\in I},\circ)$ is a category (called category of groups).
		\item Let $(M_i)_{i\in I}$ be a family of topological spaces (respectively smooth manifolds). For every $i,j\in I$, let $\hom (M_i,M_j)$ be the set of all continuous maps (resp. smooth maps) from $M_i$ to $M_j$. Then $\mathcal{C}=((M_i)_{i\in I},(\hom (M_i,M_j))_{i,j\in I},\circ)$ is a category.
	\end{itemize}
	\end{examples}	
	%%%
	%% DEF: Functor
	%%%	
	\begin{Def} Let $\mathcal{C}_1$ and $\mathcal{C}_2$ be two categories. A functor (or a functorial rule) from $\mathcal{C}_1$ to $\mathcal{C}_2$ is a pair of mappings $\mathrm{ob}(\mathcal{C}_1) \to \mathrm{ob}(\mathcal{C}_2)$ and $\hom (\mathcal{C}_1)\to \hom (\mathcal{C}_2)$, which we will denote by the same letter $F$, satisfying:
	\begin{enumerate}
	\item For every object $U$ in $\mathcal{C}_1$, $F(U)$ is an object in $\mathcal{C}_2$;
	\item For every two objects $U$ and $V$ in $\mathcal{C}_1$ and an arrow $f\in\hom_{\mathcal{C}_1} (U,V)$, $F(f)$ is in $\hom_{\mathcal{C}_2} (F(U),F(V))$;
	\item For all objects $U,V$ and $W$ in $\mathcal{C}_1$, and arrows $f\in\hom_{\mathcal{C}_1} (U,V)$ and $g\in\hom_{\mathcal{C}_1} (V,W)$ $F(g\circ f)=F(g) \circ F(f)$ and $F(\textrm{id}_{U})=\textrm{id}_{F(U)}$.
	\end{enumerate}
	\end{Def}
	%%%
	%% Examples of functors
	%%%	
	\begin{examples} \begin{itemize}
	\item Given any category, there exists a trivial functor from this category to itself preserving all objects and arrows called the identity functor.
	\item Given a category of groups $\mathcal{C}_1$, let $\mathcal{C}_2$ be a category of sets whose objects are the underlying sets to the objects of $\mathcal{C}_1$ and whose arrows are the group homomorphisms in $\mathcal{C}_1$ taken as maps between the underlying sets. The natural map associating to each object and arrow in $\mathcal{C}_1$ their set-counterparts in $\mathcal{C}_2$ is a functor. Functors constructed in a similar manner are called forgetful functors.
	\item Let $(G_i)_{i\in I}$ be a family of Lie groups. For every $i,j\in I$, let $\hom (G_i,G_j)$ denote the set of all Lie group homomorphisms from $G_i$ to $G_j$, $\mathrm{Lie}(G_i)$ denote the Lie algebra of $G_i$ and $\hom (\mathrm{Lie}(G_i),\mathrm{Lie}(G_j))$ denote the set of all Lie algebra homomorphisms from $\mathrm{Lie}(G_i)$ to $\mathrm{Lie}(G_i)$. Then 
	\[\mathcal{C}_1:=((G_i)_{i\in I},(\hom (G_i,G_j))_{i,j\in I},\circ)\]
	 and 
	 \[\mathcal{C}_2:=((\mathrm{Lie}(G_i))_{i\in I},(\hom (\mathrm{Lie}(G_i),\mathrm{Lie}(G_j)))_{i,j\in I},\circ)\] are categories. The mapping associating to each object in $\mathcal{C}_1$ its Lie algebra and to each arrow in $\mathcal{C}_1$ its pushforward at the identity defines a functor from $\mathcal{C}_1$ to $\mathcal{C}_2$:
	 \[G_i\mapsto \mathrm{Lie}(G_i);\quad \varphi\in \hom (G_i ,G_j) \mapsto \varphi_*(1_{G_i})\in \hom (\mathrm{Lie}(G_i),\mathrm{Lie}(G_j))\]
	\end{itemize}
	\end{examples}
	%%%%%
	%% A functorial rule for RP
	%%%%%
	\subsection{A functorial rule for rough paths}
	Let $E$ be a Banach space and $1 \leq p<\gamma$. Let $(U_i)_{i\in I}$  be a family of open subsets of $E$. It is obvious that 
	\[\mathcal{C}_1=((U_i)_{i\in I},(\textrm{Lip}_{\textrm{loc}}(\gamma,U_i,U_j)))_{i,j\in I}, \circ)\]
	is a category. For $i,j \in I$, we denote by $\hom (G\Omega_p(U_i),G\Omega_p(U_j))$ the set of maps that assign in a continuous way (in the $p$-rough path topology) to each rough path in $U_i$ a rough path in $U_j$ defined over the same compact interval. It is also straightforward that:
	\[\mathcal{C}_2=((G\Omega_p(U_i))_{i\in I},(\hom (G\Omega_p(U_i),G\Omega_p(U_j)))_{i,j\in I}, \circ)\]
	defines a category. Finally for any open subsets $U$ and $V$ of $E$ and $f\in \textrm{Lip}_{\textrm{loc}}(\gamma,U,V)$, we denote by $f_*$ the following map:
	\[\begin{array}{rccl}
	f_*:&G\Omega_p(U)&\longrightarrow&G\Omega_p(V)\\
	&(x,X,J)&\longmapsto&(f(x),\int{\textrm{d}f(x,X)\textrm{d}X},J)\\
	\end{array}\]
	Recall that by theorem \ref{ContIntegral}, $f_*:G\Omega_p(J ;U; E) \rightarrow G\Omega_p(J;V;E)$ is continuous for any compact interval $J$.
	\begin{Rem} In order to use the formalism of category theory, we need the identity map between two open sets to be an arrow. This is one of the reasons we introduce locally Lipschitz maps.
	\end{Rem}
	%%%
	%%  THM: Functor for RPaths
	%%%	
	\begin{theo} The rule that assigns to every object $U$ in $\mathcal{C}_1$ the object $G\Omega_p(U)$ and to every morphism $f\in \textrm{Lip}_{\textrm{loc}}(\gamma,U,V)$, where $U$ and $V$ are objects in $\mathcal{C}_1$, the map $f_*$, is functorial.
	\end{theo}
	\begin{proof} For every open subsets $U$, $V$ and $W$ of $E$ that are objects in $\mathcal{C}$, we need to prove the following:
	\begin{enumerate}
	\item $(Id_U)*=Id_{G\Omega_p(U)}$;
	\item $\forall f\in \textrm{Lip}_{\textrm{loc}}(\gamma,U,V), \forall g\in \textrm{Lip}_{\textrm{loc}}(\gamma,V, W): (g\circ f)_*=g_*\circ f_*$.
	\end{enumerate}
	The first assertion regarding the identity map is straightforward. Let $U$, $V$ and $W$ be open subsets of $E$ that are objects in $\mathcal{C}$. Let $f\in \textrm{Lip}_{\textrm{loc}}(\gamma,U,V)$ and $g\in \textrm{Lip}_{\textrm{loc}}(\gamma,V, W)$. Let $x$ be a $U$-valued path over a segment $J$ with bounded variation. By lemma \ref{IntegralSig}
	\[	f_*(x,S_{[p]}(x),J)=(f(x),\int{\textrm{d}f(x,S_{[p]}(x))\textrm{d}S_{[p]}(x)},J)= (f(x),S_{[p]}(f(x)),J)\]
	As $f(x)$ has bounded variation (lemma \ref{LipImagepvar}), we similarly have:
	\[g_*(f(x),S_{[p]}(f(x)),J)=(g\circ f(x),S_{[p]}(g\circ f(x)),J)\]
	and as $g\circ f$ is locally Lip-$\gamma$:
	\[(g\circ f(x),S_{[p]}(g\circ f(x)),J)=(g\circ f(x),\int{\textrm{d}(g\circ f)(x,S_{[p]}(x))\textrm{d}S_{[p]}(x)},J)\]
	Therefore $((g\circ f)_*)_{|G\Omega_1(U)}=(g_*\circ f_*)_{|G\Omega_1(U)}$. As both $(g\circ f)_*$ and $g_*\circ f_*$ are continuous in the $p$-variation metric (theorem \ref{ContIntegral}) and as $G\Omega_1(U)$ is dense in $G\Omega_p(U)$ for this metric we deduce that $(g\circ f)_*=g_*\circ f_*$.
	\end{proof}
	%%%%%
	%% A new definition of rough paths on manifolds
	%%%%%
	\subsection{A new definition of rough paths on manifolds}
	Based on our findings so far, we are now to able to give a minimal approach for defining rough paths on a manifold.
	%%%
	%%  DEF: Locally Lip Manifolds
	%%%	
	\begin{Def}Let $\gamma\geq1$. Let $M$ be a topological manifold and $\mathcal{A}$ an atlas on $M$. We say that $(M,\mathcal{A})$ is a locally $\textrm{Lip}-\gamma$ $n$-manifold if for any two charts $(U,\phi)$ and $(V,\psi)$ in $\mathcal{A}$ such that $U\cap V\neq\varnothing$, the map $\psi\circ\phi^{-1}:\phi(U\cap V)\rightarrow\mathbb{R}^n$ is locally $\textrm{Lip}-\gamma$.
	\end{Def}
	%%%
	%%  Example of Locally Lip Manifold
	%%%	
	\begin{examples}\begin{itemize}
	\item A Lip-$\gamma$ manifold is a locally Lip-$\gamma$ manifold. In particular, finite-dimensional vector spaces are locally Lip-$\gamma$ manifolds.
	\item Let $n\in\mathbb{N}^*$. A $\mathcal{C}^n$ manifold is a locally Lip-$n$ manifold.
	\end{itemize}
	\end{examples}
	%%%
	%%  DEF: Local RP on Lip Manifold
	%%%	
	\begin{Def}\label{RPDefII}
	Let $n\in\mathbb{N}^*$ and $p,\gamma\in [1,+\infty[$ such that $\gamma>p$. Let $M$ be a locally $\textrm{Lip}-\gamma$ $n$-manifold.  A local $p$-rough path on $M$ over a compact interval $J$ is a collection $(x_i,X_i,J_i,(\phi_i,U_i))_{i\in I}$ of $p$-rough paths on $\mathbb{R}^n$ satisfying the following conditions:
	\begin{itemize}
	\item $(J_i)_{i\in I}$ is a compact cover for $J$ with compact intervals;
	\item For every $i\in I$, $(\phi_i,U_i)$ is a locally $\textrm{Lip}-\gamma$ chart on $M$.
	\item $\forall i \in I:\quad (x_i,X_i,J_i)\in G\Omega_p(\phi_i(U_i))$;
	\item \textbf{(Consistency condition)} If $i,k\in I$ such that  $ J_i\cap J_k \neq\varnothing$, then we have:
	\[(\phi_k\circ\phi_i^{-1})_*(x_i,X_i,J_i\cap J_k)=(x_k,X_k,J_i\cap J_k)\]
	\end{itemize}
	\end{Def}
	We identify similar local rough paths in the following way:
	%%%
	%%  DEF: Local RP on Lip Manifold II
	%%%	
	\begin{Def}\label{EquivalenceLocalRP} Let $p,\gamma\in [1,+\infty[$ such that $\gamma>p$. Let $M$ be a locally $\textrm{Lip}-\gamma$ manifold and $J$ a compact interval. Two local $p$-rough paths $A=(x_i,X_i,J_i,(\phi_i,U_i))_{i\in I}$ and $B=(x_i,X_i,J_i,(\phi_i,U_i))_{i\in K}$ on $M$ over $J$ are said to be equivalent if $A\cup B$ is also a local $p$-rough path.
	\end{Def}
	This, of course, defines an equivalence relation. We will henceforth only consider the equivalence classes associated to this relation.
	%%%
	%%  DEF: RP Extension for local RP
	%%%	
	\begin{Def}\label{RPExtension}
	Let $p,\gamma\in [1,+\infty[$ such that $\gamma>p$. Let $M$ be a locally $\textrm{Lip}-\gamma$ manifold. A local $p$-rough path $(x_i,X_i,J_i,(\phi_i,U_i))_{i\in I}$ on a compact interval $J$ is said to be a $p$-rough path extension for the path $x:J\rightarrow M$ if the following holds:
	\[\forall i\in I: \quad x(J_i)\subseteq U_i \quad\textrm{ and }\quad   x_i=\phi_i \circ x_{|J_i}\]
	If such a rough path exists, we say then that $x$ admits a $p$-rough path extension.
	\end{Def}
	Following definition \ref{EquivalenceLocalRP}, we will consider that two rough path extensions of a given path to be the same if their union is a local rough path extension of that path.
	%%%%%
	%% Consistency with previous definitions
	%%%%%
	\subsection{Consistency with previous definitions}
	We make sure that our notion of rough path is compatible with the classical one:
	%%%
	%%  Lemma: Local RP => local RP
	%%%
	\begin{lemma}\label{ConsitencyLocRPRP} Let $p,\gamma\in [1,+\infty[$ such that $\gamma>p$. Let $M$ be a Lipschitz-$\gamma$ manifold and $J$ a compact interval. Let $(x_i,X_i,J_i,(\phi_i,U_i))_{i\in I}$ be a local $p$-rough path on $M$ over $J$. Then there exists a unique $p$-rough path $\mathbb{X}$ on $M$ such that:
	\[\forall i\in I: (\phi_i)_*(\mathbb{X})_{|J_i}=(x_i,X_i)\]
	\end{lemma}
	\begin{proof} By lemma \ref{LocalMultiGlobalIdII} and the definition of a rough path on a manifold, it is clear that if such a rough path exists, it is necessarily unique. Let $(a_j)_{0\leq j \leq n}$ be a subdivision of $J$ and $(i_j)_{0\leq j \leq n}$ be a finite sequence of elements of $I$ such that for all $j\in [\![0,n-1]\!]$, we have $[a_j,a_{j+1}]\subseteq J_{i_j}$. For $j\in [\![0,n-1]\!]$, we define the rough path:
	\[(y_j,Y_j)=({x_{i_j}}_{|[a_j,a_{j+1}]},{X_{i_j}}_{|[a_j,a_{j+1}]})\in G\Omega_p(U_{i_j})\]
	It is trivial then that $(y_j,Y_j,[a_j,a_{j+1}],(\phi_{i_j},U_{i_j}))_{0\leq j \leq n}$ is equivalent to $(x_i,X_i,J_i,(\phi_i,U_i))_{i\in I}$. For $j\in [\![0,n-1]\!]$, we define the p-rough path $Z_j=(\phi_{i_j}^{-1})_*(y_j(a_j),Y_j)$ on $M$ (with starting point denoted $z_j$). Let $\mathbb{X}=Z_1*\cdots Z_{n-1}$. One can check that the latter concatenation is indeed well-defined as we have, for every $j\in [\![0,n-1]\!]$ and $f$ Banach space-valued compactly supported Lip-$\gamma$ map on $M$:
	\[Z_j(f_*)^1_{a_j,a_{j+1}}=Y_j((f\circ \phi_{i_j}^{-1})_*)^1_{a_j,a_{j+1}}
	=f\circ \phi_{i_j}^{-1}(y_j(a_{j+1}))-f\circ \phi_{i_j}^{-1}(y_j(a_{j}))=f(z_{j+1})-f(z_j)\]
	By construction $\mathbb{X}$ satisfies the required conditions.
	\end{proof}
	This identification is in fact one-to-one:
	%%%
	%%  THM: Consistency of def: RP & local RP
	%%%
	\begin{theo} Let $p,\gamma\in [1,+\infty[$ such that $\gamma>p$. Let $M$ be a Lipschitz-$\gamma$ manifold. There is a one-to-one mapping between $p$-rough paths on $M$ and local $p$-rough paths on $M$.
	\end{theo}
	\begin{proof} Without loss of generality, we consider rough paths defined over the unit interval $J=[0,1]$.
	Let $\mathbb{X}$ be a $p$-rough path in $M$ over $[0,1]$ with starting point $x$. Following theorem \ref{LocalSequenceMRP}, let $(\mathbb{X}^i,x_i, (\phi_i,U_i))_{1\leq i \leq N}$ be a localising sequence for $\mathbb{X}$ adapted to a subdivision $(s_i)_{0\leq i \leq N}$ of $[0,1]$. For $i\in [\![1,N]\!]$, $(\phi_i)_*(\mathbb{X}^i)$ identifies with a local geometric $p$-rough path in $\phi_i(U_i)$ over $[s_{i-1},s_i]$ that we denote $(y_i,Y_i)$. We define the collection $Y=(y_i,Y_i,[s_{i-1},s_i],(\phi_i,U_i))_{1\leq i \leq N}$. For $i\in [\![1,N-1]\!]$, we have, first by the identification between $(y_i,Y_i)$ and $(\mathbb{X}^i,x_i)$, then the consistency of the endpoint of $\mathbb{X}^i$ with the starting point of $\mathbb{X}^{i+1}$:
	\[\begin{array}{rcl}
	(\phi_{i+1}\circ\phi_i^{-1})_*(y_i,Y_i,\{s_i\})&=&	(\phi_{i+1}\circ\phi_i^{-1})_*(y_i(s_i),1,\{s_i\})\\
								&=&	((\phi_{i+1}\circ\phi_i^{-1})(y_i(s_i)),1,\{s_i\})\\
								&=&	((\phi_{i+1}\circ\phi_i^{-1})(\phi_i(x_i)+(\mathbb{X}^i(\phi_i^{-1})_*)^1_{s_{i-1},s_i},1,\{s_i\})\\
								&=&	(\phi_i(x_{i+1}),1,\{s_i\})\\
								&=&	(y_{i+1},Y_{i+1},\{s_i\})\\
	\end{array}\]	
	This proves that $Y$ is indeed a local $p$-rough path on $M$ over $[0,1]$. Lemma \ref{ConsitencyLocRPRP} shows that the mapping $\mathbb{X} \mapsto Y$ constructed above is onto and one-to-one.	
	%invariant under the choice of the localising sequence. Let $(\mathbb{Z}^i,z_i, (\psi_i,V_i))_{1\leq i \leq r}$ be a localising sequence for $\mathbb{X}$ adapted to a subdivision $(u_i)_{0\leq i \leq r}$ of $[0,1]$ and denote $Z=((\psi_i)_*(\mathbb{Z}^i),[u_{i-1},u_i],(\psi_i,V_i))_{1\leq i \leq r}$. To show that $Y=Z$, it suffices to show the consistency of their components over overlapping segments. Let $i\in [\![1,N-1]\!]$ and $j\in [\![1,r-1]\!]$ such that $[s_{i-1},s_i]$ and $[u_{j-1},u_j]$ have a non-empty intersection, which we will denote $[a,b]$. Now note that $(\psi_j\circ\phi_i^{-1})_*(y_i,Y_i,[a,b])$ correspond to the pushforward of the rough path $(y_i(a),{Y_i}_{|[a,b]})$ by $\psi_j\circ\phi_i^{-1}$, which is equal, by the chain rule for rough paths, to the pushforward of $(\phi_i^{-1})_*(y_i(a),{Y_i}_{|[a,b]})$ by $\psi_j$. From the definition of $Y_i$ and $Z_j$ from $\mathbb{X}$ as parts of localising sequences we conclude that indeed:
	%\[(\psi_j\circ\phi_i^{-1})_*(y_i,Y_i,[a,b])=(z_i,Z_i,[a,b])\]
	 %shows that this map in 
	\end{proof}
	%%%%%
	%%%%%
	%% Coloured paths on manifolds
	%%%%%
	%%%%%
	\section{Coloured paths on manifolds}
	There is a general principle underlying the procedure followed above that can be used to define any notion of \emph{colouring} already existing on the Euclidean space to a manifold, assuming that we can find a suitable functorial rule. The lift of a path into a rough path can be indeed seen as a colouring: an extra bit of information that cannot necessarily be learned by looking at the base path only.

	We can define a notion of \emph{coloured} charts and atlas over any $n$-topological manifold provided that the transition maps are arrows in an appropriate category. We will call such a manifold a \emph{coloured manifold}.
	%%%
	%%  Def: Colored manifolds
	%%%
	\begin{Def}\label{ColouredMan} Let $n\in\mathbb{N}^*$. Let $\mathcal{C}$ be a category whose objects are all open subsets of $\mathbb{R}^n$ and for which inclusion and restriction maps are also arrows. Let $M$ be a topological manifold and $\mathcal{A}$ an atlas on $M$. We say that $(M,\mathcal{A})$ is a coloured manifold with respect to $\mathcal{C}$ if for any two charts $(U,\phi)$ and $(V,\psi)$ in $\mathcal{A}$ such that $U\cap V\neq\varnothing$, we have $\psi\circ\phi^{-1}\in \hom_{C}(\phi(U\cap V),\psi(U\cap V))$.
	\end{Def}
	We easily retrieve some familiar constructions of manifolds using this language:
	%%%
	%% EXPL: Continuous Manifold
	%%%
	\begin{example}\label{ExpleContinuousManifold} Let $n\in\mathbb{N}^*$. Let $\mathcal{C}_1$ be the category whose objects are all open subsets of $\mathbb{R}^n$ and whose arrows are continuous maps between these sets. Topological $n$-dimensional manifolds are exactly the coloured manifolds with respect to $\mathcal{C}_1$.
	\end{example}
	%%%
	%% EXPL: Smooth Manifold
	%%%
	\begin{example}\label{ExpleSmoothManifold} Let $n\in\mathbb{N}^*$. Let $\mathcal{S}_1$ be the category whose objects are all open subsets of $\mathbb{R}^n$ and whose arrows are smooth maps between these sets. Smooth $n$-dimensional manifolds are the coloured manifolds with respect to $\mathcal{S}_1$.
	\end{example}
	%%%
	%% EXPL: Locally Lip Manifold
	%%%
	\begin{example}\label{ExpleLocLipManifold} Let $n\in\mathbb{N}^*$ and $\gamma\geq 1$. Let $\mathcal{L}_1$ be the category whose objects are all open subsets of $\mathbb{R}^n$ and whose arrows are locally Lipschitz-$\gamma$ maps between these sets. Locally Lip-$\gamma$ $n$-dimensional manifolds are the coloured manifolds with respect to $\mathcal{L}_1$.
	\end{example}
	Suppose now that we have a notion of \emph{coloured paths} on $\mathbb{R}^n$ that have underlying base paths which we will call traces (rough paths are an example). Denote by $T(U)$ the sets of coloured paths whose traces lie in an open subset $U$ of $\mathbb{R}^n$. Denote by $\mathcal{C}$ a category as in definition \ref{ColouredMan} and let $\widetilde{\mathcal{C}}$ be a category whose objects are the sets of coloured paths with traces in open subsets of $\mathbb{R}^n$. We assume there exists a functor from $\mathcal{C}$ to $\widetilde{\mathcal{C}}$:
	\[U\mapsto T(U)\quad ; \quad f\mapsto f_* \]
	such that for each arrow $f$ between objects $U$ and $V$ of $\mathcal{C}$, the arrow $f_*$ between $T(U)$ and $T(V)$ associates to each coloured path $X$ on $U$ a coloured path on $V$ whose trace is the image by $f$ of the trace of $X$. We can now define coloured paths on a coloured manifold in the same manner as in definition \ref{RPDefII} and the existence of coloured path extensions for manifold-based paths as in definition \ref{RPExtension}.
	%%%
	%%  Def: Colored paths on manifolds
	%%%
	\begin{Def} Let $n\in\mathbb{N}^*$. Let $\mathcal{C}$ be a category whose objects are all open subsets of $\mathbb{R}^n$ and for which inclusion and restriction maps are also arrows.	Let $M$ be a coloured $n$-manifold w.r.t $\mathcal{C}$. A coloured path on $M$ over a compact interval $J$ is a collection $(X_i,J_i,(\phi_i,U_i))_{i\in I}$ of coloured paths on $\mathbb{R}^n$ satisfying the following conditions:
	\begin{itemize}
	\item $(J_i)_{i\in I}$ is a compact cover for $J$ by segments;
	\item For every $i\in I$, $(\phi_i,U_i)$ is in the atlas of the coloured manifold $M$.
	\item $\forall i \in I:\quad X_i\in T(\phi_i(U_i))$ and $X_i$ is defined over $J_i$;
	\item \textbf{(Consistency condition)} If $i,k\in I$ such that  $ J_i\cap J_k \neq\varnothing$, then we have:
	\[(\phi_k\circ\phi_i^{-1})_*({X_i}_{|J_i\cap J_k})={X_k}_{|J_i\cap J_k}\]
	\end{itemize}
	A coloured path $(X_i,J_i,(\phi_i,U_i))_{i\in I}$ on an interval $J$ is said to be a coloured path extension for the path $x:J\rightarrow M$ if the following holds:
	\[\forall i\in I: x(J_i)\subseteq U_i \quad \textrm{and}\quad \textrm{trace}(X_i)=\phi_i \circ x_{|J_i}\]
	If such a coloured path exists, we say then that $x$ admits a coloured path extension.
	\end{Def}
	As the study of the classical examples below will show, we emphasize that the rule linking $\mathcal{C}$ and $\widetilde{\mathcal{C}}$ need be functorial in order to have sound geometric definitions and for these definitions to make sense on their own and be consistent with the definitions of coloured paths on the Euclidean space seen now as a coloured manifold.
	%%%
	%% EXPL: Continuous Paths
	%%%
	\begin{example}[Example \ref{ExpleContinuousManifold} continued] Take $\mathcal{C}_2$ to be the category whose objects are the sets of continuous paths with values in open subsets of $\mathbb{R}^n$. The colouring map associated to an arrow $f$ in $\mathcal{C}_1$ is a map that assigns to every continuous path $x$ the path $f \circ x$. Then our new definition of a continuous path on $M$ (as a coloured path) can be identified with the classical one which relies only on the topology on $M$: every continuous path on $M$ in the classical sense can be seen as the concatenation of pushforwards of continuous paths on the Euclidean space that are consistent among themselves.
	\end{example}
	%%%
	%% EXPL: Smooth Paths
	%%%
	\begin{example}[Example \ref{ExpleSmoothManifold} continued] Take $\mathcal{S}_2$ to be the category whose objects are the sets of smooth paths valued in open subsets of $\mathbb{R}^n$. The associated functorial rule from $\mathcal{S}_1$ to $\mathcal{S}_2$ is the same as above by replacing continuity with smoothness. Finally, one can see the definitions of smooth maps in the classical sense and when described as coloured paths are equivalent.
	\end{example}
	%%%
	%% EXPL: Local RP on Manifolds
	%%%
	\begin{example}[Example \ref{ExpleLocLipManifold} continued] Let $1 \leq p <\gamma$. Let $\mathcal{L}_2$ be the category whose objects are sets of local geometric $p$-rough paths supported in the open subsets of $\mathbb{R}^n$ and whose arrows are the set of mappings between the rough path in these objects defined over the same interval in a continuous way (in the $p$-rough path topology). We showed in the previous sections that rough paths can be seen as coloured paths in this setting.
	\end{example}
	%%%%%
	%% Bibliography
	%%%%%
	%\cleardoublepage
	%\addcontentsline{toc}{chapter}{Bibliography}
	
\end{document}